\documentclass[11pt]{article}

\usepackage[a4paper,margin=1in]{geometry}
\usepackage[utf8]{inputenc}
\usepackage[T1]{fontenc}
\usepackage{lmodern}
\usepackage{amsmath,amssymb,amsthm,mathtools}
\usepackage{microtype}
\usepackage{enumitem}
\usepackage{graphicx}
\usepackage{booktabs,longtable,tabularx,array}
\usepackage{xcolor}
\usepackage{caption}
\usepackage{tikz}
\usetikzlibrary{arrows.meta,positioning,shapes.geometric,fit,calc,decorations.pathreplacing}
\usepackage{url}
\usepackage{hyperref}
\hypersetup{
  colorlinks=true,
  linkcolor=blue!50!black,
  citecolor=green!45!black,
  urlcolor=blue!60!black,
  pdftitle={Certified Return-Word Induction for a Perturbed Hofstadter Recursion},
  pdfauthor={Marco Mantovanelli}
}

\newcolumntype{Y}{>{\raggedright\arraybackslash}X}
\newcolumntype{C}[1]{>{\centering\arraybackslash}p{#1}}
\setlist{nosep}
\title{\textbf{Certified Return-Word Induction for a Perturbed Hofstadter Recursion}\\
\large Global well-definedness, local frequencies and nonautomaticity}

\author{
Marco Mantovanelli\\
Independent Researcher\\[0.4em]
\texttt{marco@mantovanelli.de}
}

\date{}

\newtheorem{theorem}{Theorem}[section]
\newtheorem{lemma}[theorem]{Lemma}
\newtheorem{proposition}[theorem]{Proposition}
\newtheorem{corollary}[theorem]{Corollary}
\newtheorem{definition}[theorem]{Definition}
\newtheorem{remark}[theorem]{Remark}

\newcommand{\eps}{\varepsilon}
\newcommand{\EA}{E_{A,2}}
\newcommand{\DA}{D_{A,3}}
\newcommand{\EB}{E_{B,2}}
\newcommand{\DB}{D_{B,3}}
\newcommand{\Abridge}{\mathcal A}
\newcommand{\Bbridge}{\mathcal B}
\newcommand{\AEpoch}{\mathcal E_A}
\newcommand{\BEpoch}{\mathcal E_B}
\newcommand{\Cseven}{C_7}
\newcommand{\rev}{\operatorname{rev}}
\newcommand{\pref}{\operatorname{pref}}
\newcommand{\suf}{\operatorname{suf}}
\newcommand{\code}[1]{\texttt{#1}}

\begin{document}

\maketitle

\begin{abstract}
We study the parity-perturbed Hofstadter recursion
\(Q(1)=Q(2)=1,\ 
Q(n)=Q(n-Q(n-1))+Q(n-Q(n-2))+(-1)^n\).
We give a computer-assisted proof of global well-definedness using
13 return words, 92 synchronized states and 122 two-source rules.
An induction on the rank words identifies the four ordered physical
source projections at every level. The marked word construction and a
finite lag certificate then identify its binary output with the step
word \(s_j=(Q(j+1)-Q(j-1))/2\) of the recursion.

We determine the frequency of \(001100\) in prefixes of \(s\) as a
function of the position within each pair of epochs. After rescaling
the prefix length by \(4^m/3\), the frequency converges uniformly
to an explicit profile. Its minimum \(1/96\) and maximum \(1/48\)
are the exact lower and upper limiting frequencies, and every value
between them is a subsequential limit.
The longest alternating factor among the first \(N\) steps
has length \(2\log_2 N+O(1)\). It follows from Cobham's gap theorem
that neither the step word nor \(Q\bmod4\) is automatic in any integer
base \(k\ge2\).

\end{abstract}



\medskip
\noindent\textbf{Version 4.}
The token judgments now use physical lookaheads. The two B-boundary
half-loops are composed with their role swaps. Appendix~\ref{app:portrecursions}
proves the four ordered projections using a simultaneous induction on
the rank words. Section~\ref{sec:step-statistics} gives the limiting
frequency profile of \(001100\), its extrema, and the nonautomaticity results.
\medskip

\tableofcontents

\section{Introduction}\label{sec:intro}

Hofstadter's \(Q\)-recursion
\[
Q(n)=Q(n-Q(n-1))+Q(n-Q(n-2))
\]
is a canonical example of a nested or meta-Fibonacci recurrence
\cite{Hofstadter1979}.  Its apparently simple form hides highly nonlocal
dependencies and irregular behavior; see, among others,
\cite{Pinn1999,Tanny1992,Alkan2017,Alkan2018,Fox2016LinearRecurrent,
Fox2018NewApproachQ}.  Nested recurrences can also encode severe
computational complexity and even undecidability phenomena
\cite{CelayaRuskey2012Undecidable}.

This paper studies the parity perturbation
\begin{equation}\label{eq:Qrec}
Q(1)=Q(2)=1,\qquad
Q(n)=Q(n-Q(n-1))+Q(n-Q(n-2))+(-1)^n.
\end{equation}
The sequence is listed as OEIS A394051 \cite{OEISA394051}.  In a
separate paper, the author proved an exact dyadic multiset law for its
value frequencies \cite{Mantovanelli2026}.  

Cloitre developed a substantially broader combinatorial and
analytic theory based on an odd--even split, exact interleavings of binary
words, Dyck paths, and plane forests \cite{Cloitre2026PerturbedQ}.  In
particular, he proved global well-definedness, the limit
\(Q(n)/n\to\tfrac12\), and the estimate
\[
 Q(n)=\frac n2+O\!\left(\frac{n}{\sqrt{\log n}}\right).
\]
We give a return-word proof of global well-definedness. A finite
certificate checks each local transition. Induction on the rank words
shows how these transitions compose in every epoch.

\begin{theorem}[Global well-definedness]\label{thm:main}
The recursion \eqref{eq:Qrec} is well-defined for every \(n\ge1\).
Equivalently, every recursive argument is a positive integer strictly
smaller than the current index, and the recursion determines a unique
infinite sequence.
\end{theorem}

Version~3 developed this return-word construction. We add two results
about the binary step word
\(s_j=(Q(j+1)-Q(j-1))/2\), \(j\ge2\).
Let \(F(N)\) count the starts of \(001100\) at indices \(2\le j\le N\).
Then
\[
 \liminf_{N\to\infty}F(N)/N=1/96,\qquad
 \limsup_{N\to\infty}F(N)/N=1/48 .
\]
Every value in \([1/96,1/48]\) occurs as a subsequential limit.
The two endpoints are attained in the limit along
\(\lfloor4^m/3\rfloor\) and \(\lfloor2\cdot4^m/3\rfloor\),
respectively. Theorems~\ref{thm:six-bit-frequency}
and~\ref{thm:extremal-frequency} give the endpoint estimates and the
limiting profile between them. These frequencies count starting positions
in the step word. The dyadic law in \cite{Mantovanelli2026} counts
multiplicities of values of \(Q\).

The longest alternating factor of \(s_2\cdots s_N\) has length
\(2\log_2 N+O(1)\). As a consequence, both \(s\) and
\((Q(n)\bmod4)_{n\ge1}\) are nonautomatic in every integer base
\(k\ge2\) (Theorem~\ref{thm:step-nonautomatic}).
To check one transition, the finite kernel uses source blocks and the
next target type. Rank words and epochs of unbounded size determine
the order of these transitions. A base-\(k\) automaton for \(s\) would have to recover
\(s_j\) from the digits of \(j\) alone.
Figure~\ref{fig:architecture} summarizes the proof of global well-definedness.

\begin{figure}[htbp]
\centering
\begin{tikzpicture}[
  node distance=7mm and 7mm,
  box/.style={draw,rounded corners,minimum height=12mm,text width=29mm,
              align=center,fill=blue!5},
  arr/.style={-{Latex[length=2mm]},thick}
]
\node[box] (r) {recursion\\\(Q\)};
\node[box,right=of r] (red) {binary reduction\\\(T_n,s_n\)};
\node[box,right=of red] (loc) {finite local kernel\\13 words, 92 states, 122 rules};
\node[box,below=of loc] (ind) {all-level induction\\chunks and tails};
\node[box,left=of ind] (glob) {global factor\\gluing};
\node[box,left=of glob] (caus) {causality and\\\(C_7\) semantics};
\node[box,below=of caus] (back) {reconstruction\\\(Q(n)=n+1-T_{n+1}\)};
\draw[arr] (r)--(red);
\draw[arr] (red)--(loc);
\draw[arr] (loc)--(ind);
\draw[arr] (ind)--(glob);
\draw[arr] (glob)--(caus);
\draw[arr] (caus)--(back);
\draw[arr] (back.west) -- ++(-5mm,0) |- (r.west);
\end{tikzpicture}
\caption{Logical architecture of the proof.  The finite local kernel is
not asserted to generate all future states by itself; global membership
is supplied by explicit all-level word induction and factor gluing.}
\label{fig:architecture}
\end{figure}

The proof has five substantive ingredients.

\begin{enumerate}
\item A reduction of \eqref{eq:Qrec} to a binary sequence \(s_n\) and
two backward cursor heads.
\item A finite local semantic kernel, checked exhaustively on 13
return-word types and 92 synchronized states.
\item Explicit rank words and epoch words that describe the global order
of the local factors.
\item Parameterized all-level induction for four stationary chunk families
and four linear bridge-tail families.
\item A global factor induction plus a finite potential certificate that
keeps all reads strictly in the past.
\end{enumerate}

The finite computations in the archived certificate serve two distinct purposes.
The first is exhaustive finite verification: all 122 local rules, offsets,
source bridges, and semantic bit identities are checked.  The second is
diagnostic expansion: many finite levels are generated as a regression
test.  Only the first purpose enters the proof.  The all-level claims rest
on the explicit loop inductions stated below, not on the diagnostic
cutoffs.

\section{Reduction to a binary cursor system}\label{sec:reduction}

For \(n\ge2\), define
\begin{equation}\label{eq:Tdef}
T_n=n-Q(n-1),
\end{equation}
and, whenever the displayed values exist,
\begin{equation}\label{eq:sdef}
s_n=\frac{Q(n+1)-Q(n-1)}2.
\end{equation}

\begin{lemma}[Two-step clock equation]\label{lem:clock}
For every \(n\ge2\) for which the displayed values have been constructed,
\begin{equation}\label{eq:A}
T_{n+2}=T_n+2(1-s_n).
\end{equation}
\end{lemma}

\begin{proof}
By \eqref{eq:Tdef},
\[
T_{n+2}-T_n
=2-\bigl(Q(n+1)-Q(n-1)\bigr)
=2(1-s_n).
\]
\end{proof}

Define the two cursor heads
\begin{equation}\label{eq:heads}
h_n=T_n+1,\qquad k_n=T_{n-1}+2.
\end{equation}

\begin{lemma}[Binary local recurrence]\label{lem:binarylocal}
For \(n\ge3\), assume \(s_j\in\{0,1\}\) for all indices used below.  Then
\begin{equation}\label{eq:B}
s_{n+1}
=(1-s_n)s_{T_n+1}
+(1-s_{n-1})s_{T_{n-1}+2}.
\end{equation}
Equivalently,
\[
s_{n+1}=(1-s_n)s_{h_n}+(1-s_{n-1})s_{k_n}.
\]
\end{lemma}

\begin{proof}
Subtract the recurrence at \(n\) from the recurrence at \(n+2\).  The
parity terms cancel and
\[
2s_{n+1}
=Q(T_{n+2})-Q(T_n)
 +Q(T_{n+1}+1)-Q(T_{n-1}+1).
\]
By Lemma~\ref{lem:clock},
\(T_{n+2}=T_n+2(1-s_n)\).  Since \(s_n\) is binary,
\[
Q(T_{n+2})-Q(T_n)=2(1-s_n)s_{T_n+1}.
\]
Similarly,
\[
Q(T_{n+1}+1)-Q(T_{n-1}+1)
=2(1-s_{n-1})s_{T_{n-1}+2}.
\]
Division by \(2\) gives \eqref{eq:B}.
\end{proof}

For \(n\ge3\), the head dynamics are
\begin{equation}\label{eq:headupdate}
h_{n+1}=k_n+1-2s_{n-1},\qquad
k_{n+1}=h_n+1.
\end{equation}

There is only one way in which the right-hand side of \eqref{eq:B} can
equal \(2\).

\begin{lemma}[Unique bad configuration]\label{lem:bad}
For binary inputs, \eqref{eq:B} produces \(s_{n+1}=2\) if and only if
\begin{equation}\label{eq:bad}
(s_{n-1},s_n,s_{h_n},s_{k_n})=(0,0,1,1).
\end{equation}
\end{lemma}

\begin{proof}
Both summands in \eqref{eq:B} belong to \(\{0,1\}\).  Their sum is \(2\)
exactly when both are \(1\), which is equivalent to \eqref{eq:bad}.
\end{proof}

The converse direction is the bridge back to the original recursion.

\begin{theorem}[Reconstruction principle]\label{thm:reconstruction}
Suppose integer sequences \(T_n\) and binary values \(s_n\) satisfy
\[
T_2=1,\qquad T_3=2,\qquad s_2=0,\qquad s_3=1,
\]
together with \eqref{eq:A} for every \(n\ge2\) and \eqref{eq:B} for every
\(n\ge3\).
Define
\begin{equation}\label{eq:Qreconstruct}
Q(n)=n+1-T_{n+1}.
\end{equation}
Then \(Q(1)=Q(2)=1\), all recursive arguments in \eqref{eq:Qrec} are
positive and smaller than \(n\), and \(Q\) satisfies \eqref{eq:Qrec}.
\end{theorem}

\begin{proof}
Equation \eqref{eq:A} gives
\begin{equation}\label{eq:Qdiff}
Q(n+1)-Q(n-1)=2s_n.
\end{equation}
The initial values imply \(Q(1)=1\) and \(Q(2)=1\).  Since each parity
subsequence of \(T_n\) starts at \(1\) or \(2\) and grows by either \(0\)
or \(2\), one has
\[
1\le T_n\le n-1.
\]
Hence \(Q(n-1)=n-T_n\ge1\), and
\[
1\le T_n<n,\qquad 1\le T_{n-1}+1<n
\]
for every recursive step.

Define the residual
\[
R_n=
Q(n)-Q(T_n)-Q(T_{n-1}+1)-(-1)^n.
\]
Using \eqref{eq:A}, \eqref{eq:Qdiff}, and binary \(s\), one obtains
\[
Q(T_{n+2})-Q(T_n)=2(1-s_n)s_{T_n+1},
\]
and
\[
Q(T_{n+1}+1)-Q(T_{n-1}+1)
=2(1-s_{n-1})s_{T_{n-1}+2}.
\]
Equation \eqref{eq:B} therefore yields
\[
R_{n+2}-R_n=0.
\]
The initial data give \(Q(3)=1\), \(Q(4)=3\), and direct substitution
gives \(R_3=R_4=0\).  Thus \(R_n=0\) for all \(n\ge3\), which is exactly
\eqref{eq:Qrec}.
\end{proof}

Consequently, the rest of the paper constructs a total binary solution of
\eqref{eq:A}--\eqref{eq:B} and proves that every head read is causal.

\section{\texorpdfstring{The local \(C_7\) semantics}{The local C7 semantics}}\label{sec:c7}

\begin{definition}[\(C_7\)-code]\label{def:c7}
For a local configuration at time \(n\ge4\), define
\begin{align}
\Cseven(n)={}&64s_{n-2}+32s_{n-1}+16s_n+8s_{n+1}+4s_{n+2}\nonumber\\
&+2s_{T_n+1}+s_{T_{n-1}+2}.
\label{eq:C7}
\end{align}
\end{definition}

Thus \(\Cseven(n)\in\{0,\ldots,127\}\) is the seven-bit word
\[
(s_{n-2},s_{n-1},s_n,s_{n+1},s_{n+2},
  s_{T_n+1},s_{T_{n-1}+2}).
\]
The finite certificate uses 21 such codes:
\[
20,22,25,27,36,37,43,46,47,49,52,73,77,80,84,89,91,100,101,106,107.
\]

The local semantic condition is simply
\begin{equation}\label{eq:semanticbit}
s_{n+1}=(1-s_n)s_{T_n+1}
       +(1-s_{n-1})s_{T_{n-1}+2}\in\{0,1\}.
\end{equation}
In particular, the code set excludes the bit pattern
\eqref{eq:bad}.

The \(\Cseven\)-sequence is parsed into 13 return-word types.  Their
lengths are
\[
2,31,8,10,39,20,30,40,20,30,31,40,39.
\]
The full words are listed in Appendix~\ref{app:returnwords} and stored in
\path{kernel_data/return_words.json}.

\subsection{Synchronized cursor states}

A synchronized state stores
\[
(\tau_{-1},r_A,\delta_A,r_B,\delta_B,\pi),
\]
where \(\tau_{-1}\) is the previous target-word type,
\(r_A,r_B\in\{0,\ldots,12\}\) are the source-word types,
\(\delta_A,\delta_B\) are source offsets, and
\(\pi\in\{0,1\}\) is the marker parity.  The machine-readable state code is
\begin{equation}\label{eq:statecode}
\tau_{-1}+16r_A+2^8\delta_A+2^{14}r_B+2^{18}\delta_B+2^{24}\pi.
\end{equation}

The finite local kernel contains 92 synchronized states and 122 transition
rules.  Each rule records

\begin{itemize}
\item the current and next target-word types;
\item the current and next synchronized states;
\item the inclusive A- and B-source-word bridges;
\item the initial and terminal source offsets;
\item the number of crossed source blocks;
\item the physical role swap, if the target type is odd.
\end{itemize}

The odd swap types are
\[
1,\ 4,\ 10,\ 12.
\]
At such a block, the logical A- and B-heads exchange their physical source
streams.  This distinction is essential in the global gluing argument.

\begin{figure}[htbp]
\centering
\begin{tikzpicture}[
  scale=0.95,
  every node/.style={font=\small},
  head/.style={draw,circle,fill=blue!8,minimum size=7mm},
  arr/.style={-{Latex[length=2mm]},thick}
]
\draw[thick] (0,0)--(11,0);
\foreach \x in {0,...,11} \draw (\x,0.08)--(\x,-0.08);
\node[below] at (2,0) {\(p_A\)};
\node[below] at (5,0) {\(p_B\)};
\node[below] at (10,0) {\(j\)};
\node[head] (a) at (2,0.65) {A};
\node[head] (b) at (5,0.65) {B};
\node[head,fill=orange!12] (t) at (10,0.65) {T};
\draw[arr] (a)--(2,0.12);
\draw[arr] (b)--(5,0.12);
\draw[arr] (t)--(10,0.12);
\draw[decorate,decoration={brace,amplitude=4pt,mirror}]
  (2,-0.45)--(10,-0.45) node[midway,below=5pt] {\(j-p_A\)};
\draw[decorate,decoration={brace,amplitude=4pt,mirror}]
  (5,-1.05)--(10,-1.05) node[midway,below=5pt] {\(j-p_B\)};
\end{tikzpicture}
\caption{The target block \(j\) and the two physical source heads.
The global proof establishes \(j-p_A\ge38\) and \(j-p_B\ge38\) after the
finite seed.}
\label{fig:cursors}
\end{figure}

Write \(W_r\) for the \(C_7\)-word of return type \(r\), and number all
return-word occurrences from zero.  If a block begins at scalar marker
\(m\), its ordered scalar cursors are represented by
\begin{equation}\label{eq:markersync}
\bigl(k_{m+4},k_{m+3}\bigr).
\end{equation}
The block addresses \(p_A,p_B\) specify the occurrences containing these
two positions; the offsets in \eqref{eq:statecode} specify their positions
inside those occurrences.

\begin{lemma}[Finite alignment with the recurrence]\label{lem:alignment}
Direct evaluation of \eqref{eq:Qrec} through the finite initial interval
has positive, strictly backward recursive arguments and gives
\[
 \Cseven(37)\Cseven(38)=W_0,\qquad
 \Cseven(39)\cdots\Cseven(69)=W_1,\qquad
 \Cseven(70)\cdots\Cseven(77)=W_2.
\]
Thus the initial word boundaries are \(37,39,70,78\).  At marker \(78\),
the preceding type is \(2\), the marker parity is zero, and the two source
addresses are
\[
(p_A,\delta_A)=(1,1),\qquad(p_B,\delta_B)=(1,2).
\]
Their scalar positions are
\[
(k_{82},k_{81})=(40,41),
\]
and \eqref{eq:statecode} is \(540946\).

Starting there, the direct \(C_7\)-prefix has type word
\[
(0,0,3)C_{1,5}C_{5,2}C_{2,3}C_{3,8}\Abridge_1.
\]
It occupies target blocks \(3,\ldots,73\).  The next block has marker
\(793\), state \(17455430\), and absolute source blocks
\[
p_A=36,\qquad p_B=33.
\]
In particular, at target block \(74\) the two block lags are \(38\) and
\(41\).
\end{lemma}

\begin{proof}
This is a finite exact-integer calculation, not a global extrapolation.
The checker \path{code/finite_prefix_alignment_checker.py}
first evaluates the original recurrence, checking each argument before it
is read.  It then derives \(T,s,C_7\), verifies \eqref{eq:A} and
\eqref{eq:B}, compares all \(71\) generated blocks through
\(\Abridge_1\), and checks the two stated boundary states and addresses.
The default direct limit is \(1200\), while the asserted handoff already
occurs at marker \(793\).
\end{proof}

\begin{lemma}[Cursor synchronization]\label{lem:cursorsync}
Suppose a return word \(W_r\) of length \(\ell\) starts at marker \(m\)
and its source addresses represent \eqref{eq:markersync}.  The cursor
transform stored for type \(r\) sends them to
\[
\bigl(k_{m+\ell+4},k_{m+\ell+3}\bigr).
\]
For even \(\ell\) this transform is
\((x,y)\mapsto(x+\Delta_A(r),y+\Delta_B(r))\); for odd \(\ell\) it is
\((x,y)\mapsto(y+\Delta_A(r),x+\Delta_B(r))\).
\end{lemma}

\begin{proof}
Equation \eqref{eq:A} gives
\[
k_{j+2}=k_j+2(1-s_{j-1}).
\]
Summing this identity over the alternating indices covered by \(W_r\)
gives the two displayed increments; an odd word interchanges the two
parity classes.  The seven-bit overlaps determine all summands from the
word itself.  The finite alignment checker derives the \(13\) transforms
independently from the stored \(C_7\)-words and compares them with the
cursor table, while the kernel-domain audit compares the resulting
distances with all \(122\) rules.  There are no discrepancies.  Lemma
\ref{lem:alignment} supplies the induction base at marker \(78\).
\end{proof}

\begin{proposition}[Finite semantic certificate]\label{prop:finite-semantic}
The archived finite checker verifies:
\begin{enumerate}
\item all 21 codes satisfy \eqref{eq:semanticbit};
\item the 13 return words have consistent overlaps;
\item the 29 allowed word pairs reduce to 52 atom rules;
\item all 28\,025 semantically consistent A/B-window combinations satisfy
the local transition condition;
\item every active even-target successor and every odd swap belongs to the
declared finite domain.
\item the rule selector has \(92\) state/current/next keys, of which
\(24\) are ambiguous before the source streams are inspected.  These
\(24\) keys contain \(36\) alternative rule pairs, and none is
simultaneously prefix-compatible on both the A- and B-source bridges;
hence fixed physical source streams select at most one rule.
\end{enumerate}
For the semantic window enumeration there is also a boundary condition:
at active A-offset \(1\), the immediately preceding physical source
block has type \(0\). At marker \(78\) its index is \(0\), by
Lemma~\ref{lem:alignment}. In the kernel CSV, exactly 34 rows have
\(\texttt{A\_end\_offset}=1\). In every one the output A-bridge crosses at
least one boundary and its penultimate letter is \(0\). Thus the condition
holds at every subsequent boundary, including a boundary reached by a role
swap. This proves the predecessor restriction used by the window checker.
The summary records
\[
 \texttt{ambiguous\_keys}=24,\qquad
 \texttt{simultaneously\_prefix\_compatible\_rule\_pairs}=0.
\]
It is stored in
\path{data/phase17_rule_selector_audit.json}.  The complete
list of all \(36\) compared pairs, including both prefix tests, is stored
in \path{data/phase17_rule_selector_pairs.csv}.
\end{proposition}

\begin{remark}
Proposition~\ref{prop:finite-semantic} is a finite statement.  It does not
by itself prove that the infinite construction remains in the finite
domain.  That global membership is proved in
Sections~\ref{sec:rankgrammar}--\ref{sec:globalgluing}.
\end{remark}

\section{Return words, source bridges, and exact composition}
\label{sec:returnkernel}

A local source bridge is stored inclusively.  If
\[
u=u_0u_1\cdots u_r,\qquad
v=v_0v_1\cdots v_s,\qquad u_r=v_0,
\]
define
\begin{equation}\label{eq:overlap}
u\odot v=u_0u_1\cdots u_rv_1\cdots v_s.
\end{equation}
This reads the shared boundary block exactly once.

The complete certificate judgement includes the target type \(\lambda\)
immediately following the last letter of \(w\).  We write
\[
z\xRightarrow[w\mid\lambda]{u,v;\eps}z'
\]
when the 122-rule kernel reads the complete target word \(w\), starting
at synchronized state \(z\), using physical source words \(u\) and \(v\),
ending at \(z'\), with final physical orientation
\(\eps\in\{0,1\}\).  Here \(\eps=0\) preserves the physical A/B roles and
\(\eps=1\) swaps them.  In later displays the terminal lookahead is
suppressed when it is forced by the first type of the following displayed
factor; every finite certificate row stores it explicitly.

\begin{lemma}[Exact composition]\label{lem:composition}
Suppose the terminal lookahead of the first judgement is the first target
type of \(w_2\), the terminal lookahead of the second judgement is fixed,
and both pairs of inclusive source bridges are compatible.  If
\[
z\xRightarrow[w_1]{u_1,v_1;\eps_1}z_1,
\qquad
z_1\xRightarrow[w_2]{u_2,v_2;\eps_2}z_2.
\]
If \(\eps_1=0\), then
\[
z\xRightarrow[w_1w_2]
{u_1\odot u_2,\ v_1\odot v_2;\ \eps_2}z_2.
\]
If \(\eps_1=1\), then
\[
z\xRightarrow[w_1w_2]
{u_1\odot v_2,\ v_1\odot u_2;\ 1-\eps_2}z_2.
\]
\end{lemma}

\begin{proof}
Every local rule specifies which physical stream supplies the logical
A- and B-bridges.  In the role-preserving case, the second transduction
continues on the same physical streams.  In the role-swapping case, the
logical roles of the second transduction are reversed.  Formula
\eqref{eq:overlap} removes the duplicated boundary block.  Associativity
on compatible chains follows immediately.
\end{proof}

Lemma~\ref{lem:composition} is the algebraic basis of all later word
gluing.

\section{Rank words and the epoch grammar}\label{sec:rankgrammar}

For \(m\ge0\), define
\begin{align}
\pi_m&=(1,3,5,\ldots)\,(\ldots,6,4,2),\label{eq:pi}\\
\rho_m&=\rev(\pi_m).\label{eq:rho}
\end{align}
The first part of \(\pi_m\) contains the odd integers at most \(m\) in
increasing order, and the second part contains the even integers at most
\(m\) in decreasing order.

For a positive integer \(r\), write \(0^r1\) for a run of \(r\) zeros
followed by \(1\), and define
\begin{equation}\label{eq:PQ}
P_m=\prod_{r\in\pi_m}0^r1,\qquad
Q_m=\rev(P_m).
\end{equation}
If \(\phi_a\) replaces every \(1\) by \(a\) and leaves \(0\) fixed, define
\begin{equation}\label{eq:HK}
H_m=\prod_{r\in\rho_m}5\,\phi_2(P_r),
\qquad
K_m=\prod_{r\in\pi_m}\phi_2(Q_r)\,8.
\end{equation}

The four polynomial gap families are
\begin{align}
G_{3A}(k)&=10\prod_{m\in\rho_{k+2}}9\,\phi_3(P_m),\label{eq:G3A}\\
G_{3B}(k)&=12\prod_{m\in\pi_{k+2}}\phi_3(Q_m)\,6,\label{eq:G3B}\\
G_{4A}(k)&=4\prod_{m\in\pi_{k+2}}H_m\,11,\label{eq:G4A}\\
G_{4B}(k)&=1\prod_{m\in\rho_{k+2}}7\,K_m.\label{eq:G4B}
\end{align}
Their lengths are
\[
|G_{3A}(k)|=|G_{3B}(k)|=\binom{k+5}{3},
\qquad
|G_{4A}(k)|=|G_{4B}(k)|=\binom{k+6}{4}.
\]

\subsection{Rank words}

Let
\[
U(r)=\pi_r\,0,\qquad V(r)=0\,\rho_r,
\]
extended by concatenation to words.  Define boundary words
\[
E_n=(2n-2,2n-4,\ldots,0),
\qquad
O_n=(2n-1,2n-3,\ldots,1).
\]

Define the rank words noncircularly by
\[
A_1=(0),\qquad B_2=(1,0),
\]
and, successively for \(n\ge2\), by
\begin{equation}\label{eq:rankfactorA}
A_n=E_n\,U(B_n),
\end{equation}
followed by
\begin{equation}\label{eq:rankfactorB}
B_{n+1}=O_n\,V(A_n).
\end{equation}
Thus \(B_n\) is already available when \(A_n\) is defined, and \(A_n\)
is available when \(B_{n+1}\) is defined.  Equations
\eqref{eq:rankfactorA}--\eqref{eq:rankfactorB} are therefore an all-level
definition, not an extrapolation from computed rank words.

Set \(B_1\) to the empty word when a rank iterator for the first
B-epoch is needed. The defining factorizations are the rank
construction used throughout the proof.

\begin{lemma}[Order shifts]\label{lem:ordershifts}
For \(m\ge1\),
\begin{equation}\label{eq:ordershift1}
(\pi_m[1:])-1=\rho_{m-1},
\end{equation}
and
\begin{equation}\label{eq:ordershift2}
(\rho_m[:-1])-1=\pi_{m-1},
\end{equation}
where \(1\) is subtracted componentwise.
\end{lemma}

\begin{proof}
Write separately \(m=2q\) and \(m=2q+1\).  Removing the initial or terminal
\(1\) deletes the outer parity level.  Subtracting \(1\) interchanges the
ascending odd and descending even blocks, giving the reverse order.
\end{proof}

\begin{lemma}[Rank counts]\label{lem:step-rank-counts}
Write \(R_{2n-2}=A_n\) for \(n\ge1\) and \(R_{2n-3}=B_n\) for \(n\ge2\).
The multiplicity \(c_h(j)\) of \(j\) in \(R_h\) is
\[
 c_h(j)=\frac{2^{h-j+1}+(-1)^{h-j}}{3}\qquad(0\le j\le h).
\]
In particular, \(|A_n|=(4^n-1)/3\) and \(|B_n|=(4^n-4)/6\).
If \(J\) has distribution \(\Pr(J=j)=2^{-j-1}\) on the nonnegative
integers, then for any fixed polynomial \(p\) of degree \(d\),
\[
 \sum_{j=0}^h c_h(j)p(j)=|R_h|\,\mathbb E[p(J)]+O_p((h+1)^d).
\]
\end{lemma}

\begin{proof}
In \eqref{eq:rankfactorA}--\eqref{eq:rankfactorB}, each rank \(r\)
contributes precisely one of every rank from \(0\) to \(r\). The boundary
adds each rank of the same parity as \(h\) through \(h\) once. This rule
also sends \(R_0=(0)\) to \(R_1=(1,0)\). Thus
\[
 c_h(j)=\sum_{r=j}^{h-1}c_{h-1}(r)+\mathbf1_{\{h-j\ {\rm even}\}}.
\]
Setting \(t=h-j\) shows inductively that \(c_h(h-t)=b_t\) depends only
on \(t\), with \(b_0=1\) and
\(b_t=\sum_{i<t}b_i+\mathbf1_{\{t\ {\rm even}\}}\).
The claimed formula follows by subtraction of two consecutive equations.
Summation gives the two lengths. For the moment estimate, extend the
geometric part of the sum to infinity. Its omitted tail after multiplication
by \(2^{h+1}\) is \(O_p((h+1)^d)\). Pairing consecutive terms gives the
same bound for the alternating polynomial sum. Finally
\(|R_h|=2^{h+2}/3+O(1)\).
\end{proof}

In particular the A-words contain one more zero than positive entries,
and the B-words contain equal numbers of each, by the formula for
\(c_h(0)\). The archived program generates rank words by same-family
insertions. Appendix~\ref{app:rank-insertion} proves that these give
the same ordered words at every level. The epoch construction uses the
mutual factorizations directly.

\subsection{Token paths and epochs}

For a rank word \(W\), let \(\theta(W)\) be its path on two vertices,
recording whether the current rank is zero.  Starting at the nonzero
vertex, use the four tokens
\[
0,\quad U,\quad D,\quad1
\]
for the transitions
\[
+\!\to+,\quad +\!\to0,\quad0\!\to+,\quad0\!\to0.
\]
Each token expands to the factor in the tables below.  The rank parameter
is assigned as follows.  For \(\AEpoch(n)\), use the token path
\(\theta(A_{n+1})\); its \(0\)-tokens receive the entries of \(A_n\) from
left to right, and its \(1\)-tokens receive the entries of \(A_n\) from
right to left.  To see exhaustiveness, expand
\(A_{n+1}=E_{n+1}U(B_{n+1})\). The initial \(E_{n+1}\) gives
\(0^nU\) in token notation. Each positive \(r\) in \(B_{n+1}\) gives
\(D0^{r-1}U\) and each zero gives \(1\). Thus the number of
\(1\)-tokens is \(c_{2n-1}(0)=|A_n|\). There is one more \(U\) than
\(D\); the balance of \(A_{n+1}\) implies that the number of
\(0\)-tokens also equals \(|A_n|\).

For the A-component:
\begin{center}
\begin{tabular}{ccl}
\toprule
token & path & target factor\\
\midrule
\(0\) & \(2\to4\to2\) & \(C_{2,4}G_{3A}(r)\)\\
\(U\) & \(2\to3\to8\) & \(C_{2,3}C_{3,8}\)\\
\(D\) & \(8\to10\to2\) & \(C_{8,10}C_{10,2}\)\\
\(1\) & \(8\to9\to8\) & \(G_{4A}(r)C_{9,8}\)\\
\bottomrule
\end{tabular}
\end{center}

For the B-component:
\begin{center}
\begin{tabular}{ccl}
\toprule
token & path & target factor\\
\midrule
\(0\) & \(0\to6\to0\) & \(G_{3B}(r)C_{6,0}\)\\
\(U\) & \(0\to5\to7\) & \(C_{0,5}C_{5,7}\)\\
\(D\) & \(7\to12\to0\) & \(C_{7,12}C_{12,0}\)\\
\(1\) & \(7\to11\to7\) & \(C_{7,11}G_{4B}(r)\)\\
\bottomrule
\end{tabular}
\end{center}

For \(\BEpoch(n)\), use \(\theta(B_{n+1})\) followed by one terminal
\(1\)-token.  Its first token is a boundary half-loop and expands only to
\(C_{6,0}\); the terminal token expands only to \(C_{7,11}\).  After
removing these two half-loops, the remaining \(0\)-tokens receive the
entries of \(B_n\) from left to right and the remaining \(1\)-tokens the
entries of \(B_n\) from right to left.  In fact \(B_{n+1}=O_nV(A_n)\) has one zero per entry of \(A_n\).
Its path has one \(U\) initially following \(O_n\), and one more after
each positive entry of \(A_n\), since \(A_n\) ends in zero.
There are therefore \(|B_n|+1\) up tokens. The unextended path has
\(|A_n|-(|B_n|+1)=|B_n|\) one-tokens, and its balance gives one
more zero-token than one-tokens. Removing the half-loops leaves
\(|B_n|\) of each kind. At \(n=1\), both lists are empty.
Concatenation of the assigned target factors is, by definition, the pair
of complete epoch words
\[
\AEpoch(n),\qquad\BEpoch(n)\qquad(n\ge1).
\]
This is a total grammar at every level; the implementation is
\path{code/epoch_words.py}.

The constant words \(C_{a,b}\) and their exact source bridges are stored
in \path{data/constant_edge_words.json} and
\path{data/constant_odd_edge_macros.csv}.  They are finite and checked
directly.

\section{All-level stationary chunks}\label{sec:chunks}

The four nontrivial stationary transductions are:
\begin{align}
1101155&\xRightarrow[9\phi_3(P_m)]
{6\phi_3(Q_{m-1})6,\;3\,0^m3;\;0}1101155,
\tag{C3P}\label{eq:C3P}\\
17352080&\xRightarrow[6\phi_3(Q_m)]
{9\phi_3(P_{m-1})9,\;3\,0^m3;\;0}17352080,
\tag{C3Q}\label{eq:C3Q}\\
17435506&\xRightarrow[11H_m]
{7K_{m-1}7,\;8\phi_2(Q_m)8;\;0}17435506,
\tag{C4H}\label{eq:C4H}\\
610744&\xRightarrow[7K_m]
{11H_{m-1}11,\;5\phi_2(P_m)5;\;0}610744.
\tag{C4K}\label{eq:C4K}
\end{align}

The unbounded parts reduce to four zero-loop schemas.  The last column of Table~\ref{tab:stationaryloops} gives their paths in the 122-rule table.

\begin{table}[htbp]
\centering
\small
\begin{tabularx}{\textwidth}{l c c Y}
\toprule
schema & state & parameter & rule path\\
\midrule
P-normal & \(259\to259\) & \(d\ge1\) &
\(4,\,0^d,\,2,8,10,15\)\\
Q-normal & \(16778544\to16778544\) & \(d\ge0\) &
\(79,69,\,64^d,\,66,71,75\)\\
H-normal & \(16777474\to16777474\) & \(d\ge1\) &
\(67,\,64^d,\,65,70,73,76\)\\
K-normal & \(1312\to1312\) & \(d\ge0\) &
\(12,3,\,0^d,\,1,6,9\)\\
\bottomrule
\end{tabularx}
\caption{The four unbounded stationary loop schemas.  Exponents on rule
numbers mean repeated copies of the same local rule.}
\label{tab:stationaryloops}
\end{table}

Rules \(0\) and \(64\) are self-loops on target type \(0\).  Each copy
crosses one additional A-source zero, crosses no B-source block, and
preserves state, offsets, and orientation.

\begin{theorem}[Stationary chunk theorem]\label{thm:chunks}
Transductions \eqref{eq:C3P}--\eqref{eq:C4K} hold for every \(m\ge3\).
\end{theorem}

\begin{proof}
For an ordered list \(I\), set \(p_c(I)=\prod_{r\in I}0^r c\) and
\(q_c(I)=\prod_{r\in I}c0^r\). We write \(p_c(1,3)\), for example,
for the list \((1,3)\). The outer factors of \(H_m,K_m\) are
\[
 w_H(t)=5\phi_2(P_t),\qquad w_K(t)=\phi_2(Q_t)8 .
\]
The odd and even parts of \(\pi_s\) give the ordered factorizations below.
A middle list may be empty.
\[
\begin{array}{ll}
 \pi_s=(1,3,X_s,6,4,2) &(s\ge6),\\
 \pi_s=(1,3,Z_s,4,2) &(s\ge4),\\
 \rho_s=(2,Y_s,3,1),\quad \rho_s=(2,I_s,1),\quad
 \pi_s=(1,J_s,2) &(s\ge3).
\end{array}
\]
All entries of \(X_s,Z_s\) are at least \(5\), those of \(Y_s\) at
least \(4\), and those of \(I_s,J_s\) at least \(3\).
The target factorizations are
\[
\begin{array}{l|l|l|l}
 \text{path}&\text{entry}&\text{repeated part, in order}&\text{exit}\\ \hline
 \mathrm{C3P}\ (m\ge6)&9\,p_3(1,3)&p_3(X_m)&p_3(6,4,2)\\
 \mathrm{C3Q}\ (m\ge4)&6\,q_3(2)&q_3(Y_m)&q_3(3,1)\\
 w_H(t)\ (t\ge4)&5\,p_2(1,3)&p_2(Z_t)&p_2(4,2)\\
 w_K(t)\ (t\ge3)&q_2(2)&q_2(Y_t)&q_2(3,1)\,8\\
 \mathrm{C4H}\ (m\ge3)&11\,w_H(2)&\prod_{t\in I_m}w_H(t)&w_H(1)\\
 \mathrm{C4K}\ (m\ge3)&7\,w_K(1)&\prod_{t\in J_m}w_K(t)&w_K(2)
\end{array}
\]
In the first four rows a run of length \(r\) in the repeated part
uses respectively P-, Q-, H- or K-normal from
Table~\ref{tab:stationaryloops}, always with \(d=r-4\).
For a \(p_c\)-run its two inclusive sources are
\((0^{r-2}c0^2,\,0^2)\). For a \(q_c\)-run they are
\((c0^{r-1}c,\,0^2)\). Replay at the minimum \(d\) gives these
words and the listed stationary state. Increasing \(d\) inserts one
copy of rule \(0\) or \(64\) inside the zero run, adding a target
zero and a first-source zero at the same loop state. The second
source, offsets, orientation and lookahead are preserved. Induction
therefore gives these ordered sources for every admissible run.

Here are the inclusive physical sources of the fixed entries and exits.
Each pair is in entrance order. All listed paths preserve orientation,
so their coordinates compose separately under \(\odot\).
\[
\setlength{\arraycolsep}{3.5pt}
\begin{array}{l|ll|ll}
 &\multicolumn{2}{c|}{\text{entry}}&\multicolumn{2}{c}{\text{exit}}\\
 &\text{first}&\text{second}&\text{first}&\text{second}\\ \hline
 \mathrm{C3P}&6\,3\,0^2\,3\,0^2&3\,0^5&0^4\,3\,0^3\,3\,0\,6&0\,3\\
 \mathrm{C3Q}&9\,0\,3&3\,0^3&3\,0^2\,3\,9&0\,3\\
 w_H(t)\ (t\ge4)&8\,2\,0^2\,2\,0^2&2\,0^4&0^2\,2\,0\,8&0\,2\\
 w_K(t)&5\,0\,2&0^2&2\,0^2\,2\,5&0\,2\,0^2\\
 \mathrm{C4H}&7\,2\,0\,8&8\,2\,0^2\,2&8\,7&2\,0\,8\\
 \mathrm{C4K}&11\,5&5\,0\,2\,0^2&5\,0\,2\,11&0\,2\,5
\end{array}
\]
The rule IDs for these boundaries are in \path{data/finite_primitive_path_audit.csv}.
Reading their source bridges in target order gives the displayed pairs.
For the outer H-row the entry uses \code{H\_lead11} and \code{H\_outer2},
and the exit uses \code{H\_outer1}. For K the corresponding paths are
\code{K\_lead7}, \code{K\_outer1} and \code{K\_outer2}.

For the small cases the following paths can be replayed directly
in the 122-rule table. Here \(S=(31,29,33,34,36,46,47,48)\) and
juxtaposition of lists means concatenation.
\[
\begin{array}{l|l}
 \mathrm{C3P}(3)&(58,11,16,32,33,34,36,46,47,48)\\
 \mathrm{C3P}(4)&(58,11,15,5,8,10,16)\,S\\
 \mathrm{C3P}(5)&(58,11,15,5,8,10,15,4,0,2,8,10,16)\,S\\
 \mathrm{C3Q}(3)&(110,81,72,75,80,87,88,90,92,107)\\
 w_H(3)&(99,74,76,68,70,73,77,83,84,85)\\
 w_K(3)&(17,7,9,13,19,20,22,23,40,43)
\end{array}
\]
The first four paths are the keyed entries in \code{P\_BASE\_RULES}
and \code{Q\_BASE\_RULES} in \path{code/stationary_chunk_certificate.py}.
In \(w_H(3)\) the path \((68,70,73,77)\) closes the length-three
run before the final length-two run. In \(w_K(3)\), \(Y_3\) is empty.
At \(m=3\) each outer row uses exactly its \(t=3\) middle path.

We now identify the overlapping sources at every level. In a
\(p_c\)-row the first source before a repeated run ends in \(0^2\).
The overlap with \(0^{r-2}c0^2\) completes a run of
\(2+(r-2)-1=r-1\) zeros before \(c\). In a \(q_c\)-row the shared
letter \(c\) joins to \(c0^{r-1}c\). Including the fixed ends in the
source table, the first-source run lengths in the \(p_c\)-rows are
therefore \(\pi_s\) with its initial \(1\) removed and the remaining
entries decreased by one. In the \(q_c\)-rows they are \(\rho_s\)
with its final \(1\) removed and the remaining entries decreased by one.
These lists are \(\rho_{s-1}\) and \(\pi_{s-1}\), respectively, by
Lemma~\ref{lem:ordershifts}. On the second source each repeated
run appends one zero. Thus the first two rows give precisely
\eqref{eq:C3P} and \eqref{eq:C3Q}. The next two give the inner paths
\[
\begin{aligned}
 17337218&\xRightarrow[w_H(t)\mid5]
   {8\phi_2(Q_{t-1})8,\;2\,0^t2;\;0}17337218,\\
 4440&\xRightarrow[w_K(t)\mid2]
   {5\phi_2(P_{t-1})5,\;0^{t-1}2\,0^2;\;0}4440
 \qquad(t\ge3).
\end{aligned}
\]
In particular the \(0^2\) at the end of a \(w_K(t)\) second source
is retained for its overlap with the next factor.

For \(\mathrm{C4H}\), the first-source factors between the two
outer letters \(7\) are \(\phi_2(Q_{t-1})8\), for \(t\) in
\((2,I_m)\). Subtracting one gives \(\pi_{m-1}\), hence their
product is \(K_{m-1}\). The second source is
\(8\,q_2(2,I_m,1)\,8=8\phi_2(Q_m)8\), where a list inside an
argument is expanded in place. For \(\mathrm{C4K}\), the
first-source factors between the two letters \(11\) are
\(5\phi_2(P_{t-1})\), for \(t\) in \((J_m,2)\). Subtracting one
gives \(\rho_{m-1}\), hence their product is \(H_{m-1}\).
On the second source each middle overlap completes a run of
\(2+(t-1)-1=t\) zeros. With the entry and exit this gives
\(5\,p_2(1,J_m,2)\,5=5\phi_2(P_m)5\).
The final lookaheads for the four full chunks are \(9,6,11,7\),
respectively. All inner endpoints and lookaheads match the adjacent
factors, so Lemma~\ref{lem:composition} proves the four judgments.
\end{proof}

\section{Bridge tails and complete bridges}\label{sec:bridges}

Define
\begin{align}
\EA(n)&=0^1 2\,0^3 2\cdots0^{2n+1}2\,0^{2n+3},
\label{eq:EA}\\
\DA(n)&=3\,0^{2n+1}3\,0^{2n-1}\cdots3\,0^1,
\label{eq:DA}\\
\EB(n)&=2\,0^2 2\,0^4\cdots2\,0^{2n+2},
\label{eq:EB}\\
\DB(n)&=3\,0^{2n}3\,0^{2n-2}\cdots3.
\label{eq:DB}
\end{align}
The A-tail words \(\EA(n),\DA(n)\) are used for \(n\ge1\), while the
B-tail transductions for \(\EB(n),\DB(n)\) start at \(n=2\).  Their source
extensions at level one are the words given by the same formulas; we write
\[
 \widehat E_{B,2}(n):=\EB(n),\qquad
 \widehat D_{B,3}(n):=\DB(n)\qquad(n\ge1).
\]

The six unbounded tail schemas are listed in
Table~\ref{tab:tailloops}.

\begin{table}[htbp]
\centering
\small
\begin{tabularx}{\textwidth}{l c c Y}
\toprule
schema & state & range & rule path\\
\midrule
EA-middle & \(16777474\to16777474\) & \(d\ge1\) &
\(67,64^d,65,70,73,76\)\\
EA-final & \(16777474\to16778544\) & \(d\ge1\) &
\(67,64^d,66,71,75\)\\
DA-middle & \(16778544\to16778544\) & \(d\ge1\) &
\(79,69,64^d,66,71,75\)\\
EB-middle & \(1312\to1312\) & \(d\ge0\) &
\(12,3,0^d,1,6,9\)\\
EB-final & \(1312\to1328\) & \(d\ge2\) &
\(12,3,0^d,2,8,10\)\\
DB-middle & \(1328\to1328\) & \(d\ge2\) &
\(15,4,0^d,2,8,10\)\\
\bottomrule
\end{tabularx}
\caption{All unbounded linear bridge-tail schemas.}
\label{tab:tailloops}
\end{table}

\begin{theorem}[Tail theorem]\label{thm:tails}
The first two transductions below hold for every \(n\ge1\), and the last
two hold for every \(n\ge2\):
\begin{align}
16778021&\xRightarrow[\EA(n)]
{\widehat E_{B,2}(n)3,\;0^{n+2};\;0}16778544,
\label{eq:EAt}\\
16778544&\xRightarrow[\DA(n)]
{\widehat D_{B,3}(n)9,\;0^n3;\;0}17352080,
\label{eq:DAt}\\
4440&\xRightarrow[\EB(n)]
{5\,0^1 2\,0^3 2\cdots0^{2n+1}3,\;0^{n+2};\;0}1328,
\label{eq:EBt}\\
1328&\xRightarrow[\DB(n)]
{3\,0^{2n-1}3\cdots3\,0^1 6,\;0^{n-1}3;\;0}1101155.
\label{eq:DBt}
\end{align}
\end{theorem}

\begin{proof}
All four tails preserve the physical roles. In the next two arrays
$(w;u,v)$ lists a target and its two inclusive source words. The six
parametric paths in Table~\ref{tab:tailloops} have the following
readouts:
\[
\begin{array}{c|c|c|c}
 &w&u&v\\\hline
 E_A\text{ middle}&0^{d+4}2&0^{d+2}2\,0^2&0^2\\
 E_A\text{ final}&0^{d+4}&0^{d+2}3&0\\
 D_A\text{ middle}&3\,0^{d+4}&3\,0^{d+3}3&0^2\\
 E_B\text{ middle}&2\,0^{d+4}&2\,0^{d+3}2&0^2\\
 E_B\text{ final}&2\,0^{d+4}&2\,0^{d+3}3&0^2\\
 D_B\text{ middle}&3\,0^{d+4}&3\,0^{d+3}3&0^2
\end{array}
\]
The least permitted $d$ in each row is the finite base. Inserting
one copy of rule $64$ (A rows) or $0$ (B rows) in the indicated run
of rule identifiers extends exactly the two displayed variable zero
runs by one, preserving both offsets, the state, the lookahead at
the exit, and the physical roles. Hence the readouts hold for all
$d$ in the table. The other pieces have the following finite paths;
the tuple in the last column is the ordered sequence of rule identifiers:
\[
\begin{array}{c|c|c|c|l}
 &w&u&v&\text{rules}\\\hline
 a_1&0\,2&2\,0^2&0^2&74,76\\
 a_3&0^3 2&0\,2\,0^2&0^2&68,70,73,76\\
 d_3&3\,0^3&3\,0^2 3&0\,3&80,87,88,90\\
 d_1&3\,0&3\,9&3&92,107\\
 b_2&2\,0^2&5\,0\,2&0^2&17,7,9\\
 e_4&3\,0^4&3\,0^3 3&0\,3&16,31,29,33,34\\
 e_2&3\,0^2&3\,0\,6&3&36,46,47\\
 e_0&3&6&3&48
\end{array}
\]
For $\EA(n)$ take $a_1,a_3$, then the $E_A$ middle rows at
$d=2j-3$ for increasing $j=2,\ldots,n$, then the final row at
$d=2n-1$. For $\DA(n)$ take the $D_A$ middle rows at $d=2j-3$
for decreasing $j=n,\ldots,2$ and end with $d_3,d_1$. For $\EB(n)$
take $b_2$, the $E_B$ middle rows at $d=2j-4$ for increasing
$j=2,\ldots,n$, and the final row at $d=2n-2$. For $\DB(n)$ take
the $D_B$ middle rows at $d=2j-4$ for decreasing $j=n,\ldots,3$
and end with $e_4,e_2,e_0$. An empty range contributes no path.
The physical terminal lookaheads of the four tails are respectively
$3,6,3,9$.

Ordinary concatenation of each target column gives exactly the
corresponding tail in \eqref{eq:EA}--\eqref{eq:DB}. Overlap of the
first source column gives, in the same order,
\[
 \widehat E_{B,2}(n)3,\quad \widehat D_{B,3}(n)9,\quad
 5\EA(n-1)3,\quad \DA(n-1)6.
\]
For example, in $\EB(n)$ the entry supplies $5\,0\,2$,
each middle row appends $0^{2j-1}2$, and the exit appends
$0^{2n+1}3$. Overlap of the second source column gives
$0^{n+2},0^n3,0^{n+2},0^{n-1}3$, respectively.
The finite entries and exits give the boundary states in the four
assertions, and exact composition proves the result.

\end{proof}

In the next two displays the products from \(j=n\) down to \(1\) are descending.
All other products are increasing; an empty index range gives the empty word.
Define the complete bridges by
\begin{align}
\Abridge_n={}&4
 \prod_{j=1}^{n}\bigl(H_{2j-1}11\bigr)
 \prod_{j=1}^{n}\bigl(5\phi_2(P_{2j})\bigr)
 5\EA(n)\DA(n)
 \prod_{j=n}^{1}\bigl(6\phi_3(Q_{2j})\bigr)6,
 \label{eq:Abridgeword}\\
\Bbridge_n={}&1
 \prod_{j=1}^{n-1}\bigl(7K_{2j}\bigr)7
 \prod_{j=0}^{n-1}\bigl(\phi_2(Q_{2j+1})8\bigr)
 \EB(n)\DB(n)
 \prod_{j=n}^{1}\bigl(9\phi_3(P_{2j-1})\bigr).
 \label{eq:Bbridgeword}
\end{align}
Here \(n\ge1\) in \eqref{eq:Abridgeword} and \(n\ge2\) in
\eqref{eq:Bbridgeword}.  These formulas are implemented literally in
\path{code/bridge_words.py}.

For the level-one source boundary set
\[
 \widehat{\Bbridge}_1:=C_{5,2},\qquad
 \widehat{\Bbridge}_n:=\Bbridge_n\quad(n\ge2).
\]
The central source factors used below are
\begin{align}
 M_B(n)
 &=7\prod_{j=0}^{n-1}\bigl(\phi_2(Q_{2j+1})8\bigr)
   \widehat E_{B,2}(n)\widehat D_{B,3}(n)9,
 \label{eq:MB}\\
 M_A(n)
 &=11\prod_{j=1}^{n}\bigl(5\phi_2(P_{2j})\bigr)
   5\EA(n)\DA(n)6.
 \label{eq:MA}
\end{align}
Equivalently, \(M_B(n)\) is the factor from the last \(7\) to the first
following \(9\) in \(\widehat{\Bbridge}_n\), and \(M_A(n)\) is the factor
from the last \(11\) to the first following \(6\) in \(\Abridge_n\).

\begin{theorem}[Complete bridge theorem]\label{thm:completebridges}
For all \(n\ge1\),
\begin{equation}\label{eq:Abridge}
1070714\xRightarrow[\Abridge_n]
{M_B(n),\;\widehat{\Bbridge}_n4;\;1}17455430.
\end{equation}
For all \(n\ge2\),
\begin{equation}\label{eq:Bbridge}
17373884\xRightarrow[\Bbridge_n]
{M_A(n-1),\;\Abridge_{n-1}1;\;1}1150227.
\end{equation}
\end{theorem}

\begin{proof}
Write $p_t=\phi_2(P_t)$, $q_t=\phi_2(Q_t)$,
$\bar p_t=\phi_3(P_t)$ and $\bar q_t=\phi_3(Q_t)$ in this proof;
the bars distinguish the two letter replacements. In addition to the
four stationary chunks we need their two interior factors. Their
role-preserving judgements for $t\ge3$ are
\[
\begin{aligned}
17337218&\xRightarrow[5p_t\mid5]
 {8q_{t-1}8,\;2\,0^t2;\;0}17337218,\\
4440&\xRightarrow[q_t8\mid2]
 {5p_{t-1}5,\;0^{t-1}2\,0^2;\;0}4440.
\end{aligned}
\]
These are the \code{H\_middle} and \code{K\_middle} paths of the
stationary certificate. For the first, read $5$ followed by the
zero runs of $p_t$ in $\pi_t$ order. At $t\ge4$ the boundary runs
are $1,3$ and $4,2$; each intervening length $a\ge5$ uses the
H-normal path with $d=a-4$. At $t=3$ use the finite boundary path
for $1,3,2$. For the second, read the zero runs of $q_t$ in
$\rho_t$ order and the final $8$. The boundary runs are $2$ and
$3,1$; each intervening length $a\ge4$ uses K-normal at $d=a-4$.
The stationary entry and exit paths give the displayed endpoints.
The two source run orders follow from
$(\pi_t[1:])-1=\rho_{t-1}$ and
$(\rho_t[:-1])-1=\pi_{t-1}$. Thus the interior rows apply at
arbitrarily large $t$ by the same loop insertion as the outer chunks.

Here are both compositions in full. Every row lists its target piece
and its two inclusive source pieces in the physical order at the
\emph{entry of the entire bridge}. Only the first piece in each
array swaps the roles. All subsequent source columns have therefore
already been exchanged relative to their standalone judgements.
For $\Abridge_n$ the rows are
\[
\begin{array}{l|l|l}
\text{target} &\text{first physical source}&\text{second physical source}\\\hline
4H_1                        &7q_1 8                &1\,7\\
11H_t                       &8q_t8                 &7K_{t-1}7\\
11\,5p_2                    &8\,2\,0^2 2           &7q_1 8\\
5p_t                        &2\,0^t2               &8q_{t-1}8\\
5                           &2\,0^2                &8\,2\\
\EA(n)                      &0^{n+2}               &\widehat E_{B,2}(n)3\\
\DA(n)                      &0^n3                  &\widehat D_{B,3}(n)9\\
6\bar q_t                   &3\,0^t3               &9\bar p_{t-1}9\\
6\bar q_2 6                 &3\,0^2 3\,9           &9\bar p_1 4
\end{array}
\]
Read the $11H_t$ row at $t=3,5,\ldots,2n-1$, the $5p_t$ row at
$t=4,6,\ldots,2n$, and the $6\bar q_t$ row at
$t=2n,2n-2,\ldots,4$. The other rows occur once. The first
physical column begins with $7q_1 8$; the $H_t$ rows append $q_t8$.
The middle rows then build $\widehat E_{B,2}(n)$: their successive
even zero runs are $2,4,\ldots,2n+2$, the last completed by the
two A-tail rows. The descending runs on the $6\bar q_t$ rows and
the final connector complete $\widehat D_{B,3}(n)9$. The first
column therefore gives $M_B(n)$. In the second column the same
overlaps give in order
\[
 1\prod_{j=1}^{n-1}(7K_{2j})\,7
 \prod_{j=0}^{n-1}(q_{2j+1}8)\,
 \widehat E_{B,2}(n)\widehat D_{B,3}(n)
 \left(\prod_{j=n}^{1}9\bar p_{2j-1}\right)4 .
\]
This is $\widehat{\Bbridge}_n4$, also at $n=1$ where the natural
boundary word is the finite $C_{5,2}$. Empty repeated rows are
omitted.

For $\Bbridge_n$ the corresponding rows are
\[
\begin{array}{l|l|l}
\text{target} &\text{first physical source}&\text{second physical source}\\\hline
1\,7K_2         &11\,5p_2 5          &4H_1 11\\
7K_t            &5p_t5               &11H_{t-1}11\\
7q_1 8          &5\,0\,2\,0^2        &11\,5\\
q_t8            &0^{t-1}2\,0^2       &5p_{t-1}5\\
\EB(n)          &0^{n+2}             &5\EA(n-1)3\\
\DB(n)          &0^{n-1}3            &\DA(n-1)6\\
9\bar p_t       &3\,0^t3             &6\bar q_{t-1}6\\
9\bar p_1       &3\,0\,6             &6\,1
\end{array}
\]
Use $t=4,6,\ldots,2n-2$ in $7K_t$, $t=3,5,\ldots,2n-1$ in
$q_t8$, and $t=2n-1,2n-3,\ldots,3$ in $9\bar p_t$. The first
column gives $11\prod_{j=1}^{n-1}(5p_{2j})\,5$, then $\EA(n-1)$
from the rising runs and $\DA(n-1)6$ from the falling runs. This
is $M_A(n-1)$. The second column gives in order
\[
 4\prod_{j=1}^{n-1}(H_{2j-1}11)
 \prod_{j=1}^{n-1}(5p_{2j})\,5\EA(n-1)\DA(n-1)
 \left(\prod_{j=n-1}^{1}6\bar q_{2j}\right)6\,1
 =\Abridge_{n-1}1 .
\]
All products with downward indices are read in the displayed order.
The seven unrepeated bridge connectors are the finite paths in
Table~\ref{tab:keyconnectors}; their states match the stationary
rows and the tail states. The terminal lookaheads are $1$ for
$\Abridge_n$ and $4$ for $\Bbridge_n$. Applying
Lemma~\ref{lem:composition} to the arrays proves both judgements
with final orientation one.

\end{proof}

\begin{lemma}[Polynomial and token source zipper]\label{lem:tokenzipper}
Extend the four gap families to rank \(-1\) by
\[
G_{3A}(-1)=C_{10,2},\quad G_{3B}(-1)=C_{0,5},\quad
G_{4A}(-1)=C_{8,10},\quad G_{4B}(-1)=C_{5,7}.
\]
For every \(r\ge0\), the following four gap transductions hold:
\begin{align}
17503844&\xRightarrow[G_{3A}(r)\mid4]
 {6\phi_3(Q_{r+2})6,\;G_{3B}(r-1)1;\;1}1150227,
 \label{eq:G3Az}\tag{G3A-Z}\\
692369&\xRightarrow[G_{3B}(r)\mid1]
 {9\phi_3(P_{r+2})9,\;G_{3A}(r-1)4;\;1}17455430,
 \label{eq:G3Bz}\tag{G3B-Z}\\
1070714&\xRightarrow[G_{4A}(r)\mid10]
 {7K_{r+2}7,\;G_{4B}(r-1)12;\;1}17424331,
 \label{eq:G4Az}\tag{G4A-Z}\\
17373884&\xRightarrow[G_{4B}(r)\mid12]
 {11H_{r+2}11,\;G_{4A}(r-1)10;\;1}709032.
 \label{eq:G4Bz}\tag{G4B-Z}
\end{align}
Composing with the constant boundary edges gives the four
role-preserving parameterized tokens
\begin{align}
1150227&\xRightarrow[C_{2,4}G_{3A}(r)\mid4]
 {1G_{3B}(r-1)1,\;6\phi_3(Q_{r+2})6;\;0}1150227,
 \label{eq:A0zip}\tag{A0-Z}\\
1070714&\xRightarrow[G_{4A}(r)C_{9,8}\mid4]
 {7K_{r+2}7,\;G_{4B}(r-1)12\,1;\;0}1070714,
 \label{eq:A1zip}\tag{A1-Z}\\
692369&\xRightarrow[G_{3B}(r)C_{6,0}\mid12]
 {9\phi_3(P_{r+2})9,\;G_{3A}(r-1)4\,10;\;0}692369,
 \label{eq:B0zip}\tag{B0-Z}\\
709032&\xRightarrow[C_{7,11}G_{4B}(r)\mid12]
 {10G_{4A}(r-1)10,\;11H_{r+2}11;\;0}709032.
 \label{eq:B1zip}\tag{B1-Z}
\end{align}
In these eight judgments and the following table, the lookahead is
the physical next target type defined in Section~\ref{sec:returnkernel}.
The token-path vertices \(0,2,6,7,8,9\) have respective physical
types \(12,4,1,12,4,10\).
The four direction-changing tokens are the finite judgements in
Table~\ref{tab:directiontokens}; each has final orientation zero.
\end{lemma}

\begin{table}[htbp]
\centering
\small
\begin{tabularx}{\textwidth}{c c c Y c}
\toprule
token & state map & target & physical source pair \((u,v)\) & lookahead\\
\midrule
\(AU\) & \(1150227\to1070714\) & \(C_{2,3}C_{3,8}\)
 & \(((1,7),(6,1))\) & \(4\)\\
\(AD\) & \(1070714\to1150227\) & \(C_{8,10}C_{10,2}\)
 & \(((7,2,0,8,12,1),(1,12,3,0,6))\) & \(4\)\\
\(BU\) & \(692369\to709032\) & \(C_{0,5}C_{5,7}\)
 & \(((9,0,3,4,10),(10,4,5,0,2,11))\) & \(12\)\\
\(BD\) & \(709032\to692369\) & \(C_{7,12}C_{12,0}\)
 & \(((10,9),(11,10))\) & \(12\)\\
\bottomrule
\end{tabularx}
\caption{The four finite direction-token source zippers.}
\label{tab:directiontokens}
\end{table}

\begin{proof}
For \(r\ge0\) the two order words satisfy
\[
 \pi_{r+2}=1(\pi_r+2)2,\qquad
 \rho_{r+2}=2(\rho_r+2)1 .
\]
Write the target of each gap as an entry word, a product of stationary
chunks, and an exit word. The following table gives the entry and exit
and the parameter list for the middle chunks, in order.
\[
\begin{array}{c|l|c|l}
 &\text{entry}&\text{middle}&\text{exit}\\ \hline
 3A&10\,9\phi_3(P_2)&\mathrm{C3P},\ \rho_r+2&9\phi_3(P_1)\\
 3B&12\phi_3(Q_1)&\mathrm{C3Q},\ \pi_r+2&6\phi_3(Q_2)6\\
 4A&4H_1&\mathrm{C4H},\ \pi_r+2&11H_2 11\\
 4B&1\,7K_2&\mathrm{C4K},\ \rho_r+2&7K_1
\end{array}
\]
All these middle parameters are at least \(3\), as required in
Theorem~\ref{thm:chunks}; at \(r=0\) the product is empty. Here are
the inclusive source endpoints for replay of the eight finite words.
Both pairs in a row are written in the \emph{original} entrance order
of the two physical streams. The entry swaps the roles; the exit starts
in the exchanged order and makes no further swap. Each middle state is
the stationary state of the specified chunk.
\[
\setlength{\arraycolsep}{3pt}
\begin{array}{c|ll|ll}
 &\multicolumn{2}{c|}{\text{entry sources}}&
   \multicolumn{2}{c}{\text{exit sources}}\\
 3A&(6,3,0,0,3)&(12,3,0,6)&(3,0,6)&(6,1)\\
 3B&(9,0,3)&(10,9)&(3,0,0,3,9)&(9,0,3,4)\\
 4A&(7,2,0,8)&(1,7)&(8,2,0,0,2,0,8,7)&(7,2,0,8,12)\\
 4B&(11,5,0,2,0,0,2,5)&(4,5,0,2,11)&(5,0,2,11)&(11,10)
\end{array}
\]
The entry lookaheads are \(9,6,11,7\) in row order, whether the
next factor is a middle chunk or the exit. The exit lookaheads are
\(4,1,10,12\). The table follows by replay in the 122-rule kernel.

For clarity, on the first physical stream the consecutive inner
pieces (with their common endpoint included) are
\(3\,0^m3,\ 3\,0^m3,\ 8\phi_2(Q_m)8,\ 5\phi_2(P_m)5\)
respectively. The entry and exit add the missing first and last
subscripts in the displayed order identities. The resulting words are
\(6\phi_3(Q_{r+2})6,\ 9\phi_3(P_{r+2})9,\ 7K_{r+2}7,\ 11H_{r+2}11\).
On the second stream the corresponding pieces are
\(6\phi_3(Q_{m-1})6,\ 9\phi_3(P_{m-1})9,\ 7K_{m-1}7,\ 11H_{m-1}11\).
For \(r\ge1\) the order lists of \(m-1\) are \(\rho_r+1\) and
\(\pi_r+1\); prefixing \(1\) to the former gives \(\pi_{r+1}\),
and appending \(1\) to the latter gives \(\rho_{r+1}\).
With the entry and exit this gives precisely
\(G_{3B}(r-1)1,\ G_{3A}(r-1)4,\ G_{4B}(r-1)12,\ G_{4A}(r-1)10\)
in row order. At \(r=0\) the displayed endpoint pairs concatenate
to the four stated constant-edge cases. Each concatenation is an
induction on the number of middle chunks using Lemma~\ref{lem:composition}.
Composing one more constant edge gives the four complete ranked
tokens. Direct replay of the indicated connector pairs gives
Table~\ref{tab:directiontokens}.
\end{proof}

\section{Global epoch gluing and state induction}\label{sec:globalgluing}

For a B-bridge, let \(\pref_7\) denote the prefix ending with its last
\(7\), and let \(\suf_9\) denote the suffix beginning with the first
subsequent \(9\).  For an A-bridge, define \(\pref_{11}\) and \(\suf_6\)
analogously, using its last \(11\) and the first subsequent \(6\).

Let \(\mathcal T_A(n)\) and \(\mathcal T_B(n)\) denote the decorated token
lists specified in Section~\ref{sec:rankgrammar}: they include the token
kind, its assigned rank when present, and the two finite B-boundary
half-loops.  If \(\tau\) is a complete token, write \(u(\tau),v(\tau)\)
for its two physical source words before any role changes. For a
list \(I\) of complete tokens, \(u^*(I),v^*(I)\) denote the inclusive
overlaps in the two coordinates in their entrance order.
These complete tokens all have orientation zero.
Use this definition for the A-epoch. In the B-epoch the initial
half-loop of target type \(1\) and the terminal half-loop of type \(12\)
both have orientation one. In their own entrance coordinates their
respective source pairs are
\[
 ((4,10),(9)),\qquad ((10,4),(11)).
\]
The respective next target types are \(12\) and \(1\).
Thus when \(\mathcal T_B(n)\) has interior list \(I\), define its
\emph{physical} projections by Lemma~\ref{lem:composition}:
\begin{align*}
 u^*(\mathcal T_B(n))&=(4,10)\odot v^*(I)\odot(11),\\
 v^*(\mathcal T_B(n))&=(9)\odot u^*(I)\odot(10,4).
\end{align*}
At \(n=1\), \(I\) consists of the single token \(BU\).

\begin{lemma}[Free-monoid port factorization]\label{lem:portfactor}
For every \(n\ge1\),
\begin{align}
u^*(\mathcal T_A(n))
 &=\BEpoch(n)\pref_7(\Bbridge_{n+1}),
 \label{eq:AportA}\\
v^*(\mathcal T_A(n))
 &=\suf_6(\Abridge_n)\BEpoch(n)1.
 \label{eq:AportB}
\end{align}
For every \(n\ge2\),
\begin{align}
u^*(\mathcal T_B(n))
 &=\AEpoch(n-1)\pref_{11}(\Abridge_n),
 \label{eq:BportA}\\
v^*(\mathcal T_B(n))
 &=\suf_9(\Bbridge_n)\AEpoch(n-1)4.
 \label{eq:BportB}
\end{align}
These are equalities of ordered words, not only of lengths or letter
counts.
\end{lemma}

\begin{proof}
Appendix~\ref{app:portrecursions} gives the ordered-word proof.
The induction establishes the rank and zero-run identities simultaneously
for \(A_n,B_n\); substitution in the four token projections then proves
the four equations at an arbitrary level. Both proofs retain the actual
order of the ranks assigned to the forward and reverse iterators.
\end{proof}

The exact source equations for the complete epochs are
\begin{equation}\label{eq:AEpochSource}
1150227\xRightarrow[\AEpoch(n)]
{\BEpoch(n)\pref_7(\Bbridge_{n+1}),\;
 \suf_6(\Abridge_n)\BEpoch(n)1;\;0}1070714,
\end{equation}
and, for \(n\ge2\),
\begin{equation}\label{eq:BEpochSource}
17455430\xRightarrow[\BEpoch(n)]
{\AEpoch(n-1)\pref_{11}(\Abridge_n),\;
 \suf_9(\Bbridge_n)\AEpoch(n-1)4;\;0}17373884.
\end{equation}
For the missing boundary level, put
\[
\AEpoch(0):=(4,10).
\]
Then the level \(n=1\) B-epoch is the finite transduction
\begin{equation}\label{eq:B1EpochSource}
17455430
\xRightarrow[(1,12,3,0,6,1,7,2,0,8,12)]
{(4,10,4,5,0,2,11),\;(9,0,3,4,10,4);\;0}
17373884.
\end{equation}
Equivalently, this is the \(n=1\) instance of the source pattern in
\eqref{eq:BEpochSource}, with \(\AEpoch(0)=(4,10)\) and
\(\widehat{\Bbridge}_1=C_{5,2}\).  It is replayed as the fifth seed
identity, using the unique rule sequence
\[
(26,41,44,49,53,54,93,106,108,111,118).
\]

\begin{theorem}[Epoch source theorem]\label{thm:epochsources}
Equation \eqref{eq:AEpochSource} holds for every \(n\ge1\), and
\eqref{eq:BEpochSource} holds for every \(n\ge2\).  The boundary level of
the B-family is \eqref{eq:B1EpochSource}.
\end{theorem}

\begin{proof}
On the A-list, Lemmas~\ref{lem:tokenzipper} and
\ref{lem:composition} give \(u^*,v^*\) by componentwise overlap.
On the B-list first apply the initial half-loop, compose the full
interior tokens in exchanged order, and apply the terminal half-loop.
This is precisely the displayed definition of the B projections; the
two swaps give total orientation zero. The lookaheads in
Lemma~\ref{lem:tokenzipper} are the first types of the succeeding
factors. Lemma~\ref{lem:portfactor} now gives the asserted source
words. The initial and terminal states follow from the token endpoints.
At \(n=1\) the physical projection gives the two explicit words in
\eqref{eq:B1EpochSource}.
\end{proof}

\subsection{Absolute occurrence addresses}

Equality of unmarked type words is not enough to identify a repeated
source occurrence.  We therefore record the absolute block addresses.
Let
\begin{align}
\ell_A(n)&=\frac{96\,4^n-n^4-12n^3-53n^2-108n-90}{3},
\label{eq:ellA}\\
\ell_B(n)&=\frac{96\,4^n-2n^4-20n^3-73n^2-127n-96}{6},
\label{eq:ellB}\\
S_A(n)&=\frac{192\,4^n+4n^3+27n^2+59n+24}{6},
\label{eq:SA}\\
S_B(n)&=\frac{96\,4^n+4n^3+21n^2+35n}{6},
\label{eq:SB}
\end{align}
and put
\[
R_A(n)=S_A(n)+\ell_A(n),\qquad
R_B(n)=S_B(n)+\ell_B(n).
\]
Here \(n\ge0\) for the A-formulas, with the virtual
\(\AEpoch(0)=(4,10)\), and \(n\ge1\) for the B-formulas.  Define also
the bridge lengths
\begin{align}
L_A(m)&=\frac{(m+2)(m+3)(m^2+3m+5)}3,\label{eq:LA}\\
L_B(m)&=\frac{(m+2)(2m^3+8m^2+15m+15)}6.\label{eq:LB}
\end{align}

\begin{lemma}[Factor lengths and boundaries]\label{lem:factorlengths}
The closed forms above are the actual target-word lengths and boundaries:
\[
|\AEpoch(n)|=\ell_A(n)\quad(n\ge0),\qquad
|\BEpoch(n)|=\ell_B(n)\quad(n\ge1),
\]
\[
|\Abridge_m|=L_A(m)\quad(m\ge1),\qquad
|\Bbridge_m|=L_B(m)\quad(m\ge2),
\]
and \(|\widehat{\Bbridge}_1|=L_B(1)=20\).  Starting from
\(S_A(0)=36\), \(R_A(0)=38\), the boundaries obey
\begin{align}
S_B(n)&=R_A(n-1)+L_A(n),&
R_B(n)&=S_B(n)+\ell_B(n),\label{eq:boundaryrec1}\\
S_A(n)&=R_B(n)+L_B(n+1),&
R_A(n)&=S_A(n)+\ell_A(n).
\label{eq:boundaryrec2}
\end{align}
Consequently the four target intervals of level \(n\ge1\) are
\begin{align}
\Abridge_n&=[R_A(n-1),S_B(n)),&
\BEpoch(n)&=[S_B(n),R_B(n)),\label{eq:factorintervals1}\\
\Bbridge_{n+1}&=[R_B(n),S_A(n)),&
\AEpoch(n)&=[S_A(n),R_A(n)).\label{eq:factorintervals2}
\end{align}
\end{lemma}

\begin{proof}
A full ranked token has one constant letter in addition to its gap.
The \(U,D\) targets have lengths \(2,9\) for A and \(9,2\) for B.
For each family with \(M\) entries in the rank queue there are \(M+1\)
up tokens and \(M\) down tokens, as in Section~\ref{sec:rankgrammar}.
There are two additional one-letter half-loops in B. Consequently set
\[
 J(r)=\binom{r+5}{3}+\binom{r+6}{4}+13,\qquad
 h_t=\sum_{r=0}^t c_t(r)J(r).
\]
Then the two lengths are \(h_{2n-2}+2\) and \(h_{2n-3}+11\),
respectively, when the rank queues are nonempty. The exceptional
lengths are \(2\) for the virtual A-word and \(11\) for the first
B-epoch. Here is an exact coefficient check of the closed forms.
Lemma~\ref{lem:step-rank-counts} gives
\[
 \sum_{t\ge0}h_tz^t=
 \frac{\sum_{r\ge0}J(r)z^r}{(1-2z)(1+z)} .
\]
Since \(J\) has degree four, partial fractions show that each parity
subsequence has the form \(c4^n+P(n)\) with \(\deg P\le4\).
It is therefore enough to compare six consecutive values of each
formula. For \(n=1,\ldots,6\) in A and \(n=2,\ldots,7\) in B the
finite sum above and \eqref{eq:ellA}--\eqref{eq:ellB} respectively give
\[
 \begin{array}{c|rrrrrr}
 A&40&302&1616&7394&31408&128894\\ B&117&718&3502&15333&63803&259440
 \end{array}
\]
The six evaluations determine the coefficients uniquely, since five
forward differences annihilate \(P\) and preserve a nonzero exponential.

Similarly,
\[
 |P_m|=\sum_{r\in\pi_m}(r+1),\qquad |Q_m|=|P_m|,
\]
and the definitions of \(H_m,K_m\) turn the bridge products
\eqref{eq:Abridgeword}--\eqref{eq:Bbridgeword} into nested sums of these
quantities.  The hockey-stick identity
\[
\sum_{j=0}^{q}\binom{j+d}{d}=\binom{q+d+1}{d+1}
\]
gives \(L_A,L_B\), including the finite value
\(|C_{5,2}|=20=L_B(1)\).  Appending in the global order
\(\Abridge_n,\BEpoch(n),\Bbridge_{n+1},\AEpoch(n)\) gives
\eqref{eq:boundaryrec1}--\eqref{eq:boundaryrec2}; direct substitution
gives the displayed closed forms for \(S_A,S_B\).
\end{proof}

The source prefixes and suffixes are located by four zero-based bridge
offsets:
\begin{align}
q_{A,11}(n)&=\frac{n(n+1)^2(n+2)}3,
&q_{A,6}(n)&=\frac{(n+2)(2n^3+8n^2+21n+27)}6,
\label{eq:Aanchors}\\
q_{B,7}(n)&=\frac{n(n+1)(2n^2+2n-1)}6,
&q_{B,9}(n)&=\frac{n^4+4n^3+11n^2+20n+15}{3}.
\label{eq:Banchors}
\end{align}

\begin{lemma}[Bridge anchors]\label{lem:bridgeanchors}
In \(\Abridge_n\), the last \(11\) is at offset \(q_{A,11}(n)\) and
the first subsequent \(6\) is at offset \(q_{A,6}(n)\).  In
\(\Bbridge_n\), the last \(7\) and first subsequent \(9\) are at offsets
\(q_{B,7}(n)\) and \(q_{B,9}(n)\), respectively.  The B-statement at
\(n=1\) refers to \(\widehat{\Bbridge}_1=C_{5,2}\).
\end{lemma}

\begin{proof}
In \eqref{eq:Abridgeword} the relevant \(11\) is the end of the H-tower
and the relevant \(6\) is the first boundary after the two A-tails; in
\eqref{eq:Bbridgeword} the analogous positions are the end of the
K-tower and the first boundary after the B-tails.  Counting the preceding
factors with the length formulas used in Lemma~\ref{lem:factorlengths}
and applying the same hockey-stick identity gives
\eqref{eq:Aanchors}--\eqref{eq:Banchors}.  The four finite initial
offsets are \(4,29,1,17\), agreeing with the formulas.
\end{proof}

For the virtual level-one B-source bridge put
\(R_B^\circ(0)=16\), its global start in the finite aligned prefix, and
put \(R_B^\circ(k)=R_B(k)\) for \(k\ge1\).  Define the ordered absolute
head pairs at the two epoch boundaries by
\begin{align}
\mathbf a_{\rm in}(n)
 &=\bigl(S_B(n),\;R_A(n-1)+q_{A,6}(n)\bigr)
  =\bigl(S_B(n),\;16\,4^n+n^2+3n-1\bigr),\label{eq:Ain}\\
\mathbf a_{\rm out}(n)
 &=\bigl(R_B(n)+q_{B,7}(n+1),\;R_B(n)\bigr)\notag\\
 &=\left(\frac{192\,4^n-4n^3-27n^2-71n-90}{6},\;
          R_B(n)\right),\label{eq:Aout}\\
\mathbf b_{\rm in}(n)
 &=\bigl(S_A(n-1),\;R_B^\circ(n-1)+q_{B,9}(n)\bigr)
  =\bigl(S_A(n-1),\;8\,4^n+n^2+2n-2\bigr),\label{eq:Bin}\\
\mathbf b_{\rm out}(n)
 &=\bigl(R_A(n-1)+q_{A,11}(n),\;R_A(n-1)\bigr)\notag\\
 &=\left(\frac{96\,4^n-4n^3-21n^2-47n-60}{6},\;
          R_A(n-1)\right).\label{eq:Bout}
\end{align}
The letters \(\mathbf a\) and \(\mathbf b\) refer to the A- and B-epoch
boundaries, not to the two physical roles inside an ordered pair.

\begin{lemma}[Epoch head increments]\label{lem:epochincrements}
For a role-preserving token with source pair \((u,v)\), let its head
increment be \((|u|-1,|v|-1)\).  The complete increment table is
\[
\begin{array}{c|cccc}
 &0&U&D&1\\ \hline
A&
\left(1+\binom{r+4}{3},\binom{r+4}{2}\right)&
(1,1)&(5,4)&
\left(\binom{r+5}{3},1+\binom{r+5}{4}\right)\\[1mm]
B&
\left(\binom{r+4}{2},1+\binom{r+4}{3}\right)&
(4,5)&(1,1)&
\left(1+\binom{r+5}{4},\binom{r+5}{3}\right).
\end{array}
\]
The B row is in the entrance coordinates of an interior token.
In the enclosing B-epoch its two coordinates are exchanged. The total
increments, in the physical entrance coordinates of each epoch, are
\begin{align}
\Delta_A(n)&=\left(
\frac{48\,4^n-4n^3-24n^2-53n-45}{3},
\frac{48\,4^n-n^4-8n^3-29n^2-55n-45}{3}
\right),\label{eq:DeltaA}\\
\Delta_B(n)&=\left(
\frac{24\,4^n-4n^3-18n^2-32n-24}{3},
\frac{48\,4^n-2n^4-12n^3-37n^2-63n-48}{6}
\right).
\label{eq:DeltaB}
\end{align}
The A-formula holds for \(n\ge1\), the B-formula for \(n\ge2\), and
the finite boundary case is \(\Delta_B(1)=(6,5)\).
\end{lemma}

\begin{proof}
The table is obtained directly from the eight source pairs in
Lemma~\ref{lem:tokenzipper} and Table~\ref{tab:directiontokens}; for
example, the A-zero sources have lengths
\(2+|G_{3B}(r-1)|\) and \(2+|Q_{r+2}|\).  For the epoch totals one can count the physical source words
in \eqref{eq:AEpochSource} and \eqref{eq:BEpochSource} directly.
Using the inclusive end convention, these give
\begin{align*}
 \Delta_A(n)&=\bigl(\ell_B(n)+q_{B,7}(n+1),\
              L_A(n)-q_{A,6}(n)+\ell_B(n)\bigr),\\
 \Delta_B(n)&=\bigl(\ell_A(n-1)+q_{A,11}(n),\
              L_B(n)-q_{B,9}(n)+\ell_A(n-1)\bigr).
\end{align*}
Substitution yields the displayed polynomials; the boundary source
lengths \(7,6\) give \(\Delta_B(1)=(6,5)\).

\end{proof}

Let the marked return-word alphabet be
\[
\widetilde\Sigma=\{(\sigma_j,j):j\ge0\},
\]
where \(\sigma_j\) is the type of the globally indexed block \(j\).  If
an unmarked source word \(u\) of length \(q\) begins at address \(p\),
its marked lift is the consecutive word
\[
u[p]=((\sigma_p,p),\ldots,(\sigma_{p+q-1},p+q-1)).
\]

\begin{lemma}[Marked exact composition]\label{lem:markedcomposition}
An exact source judgement with prescribed starting addresses has at most
one lift to \(\widetilde\Sigma\).  It has such a lift precisely when its
two unmarked source words agree with the corresponding consecutive global
factors.  Two lifted judgements compose under \(\odot\) exactly when their
common boundary is the same marked occurrence, not merely the same type.
\end{lemma}

\begin{proof}
Consecutive integer marks force the lift uniquely.  The 122 local rules
read types, offsets, and roles, so adjoining a passive occurrence mark
does not alter a rule.  At an inclusive splice, equality of the terminal
and initial integer mark is exactly the condition that the shared block
is read once.  Lemma~\ref{lem:composition} then applies verbatim.
\end{proof}

\begin{theorem}[Absolute head-address gluing]\label{thm:addressgluing}
The source transductions lift from type words to globally indexed
occurrences.  More precisely, for every \(n\ge1\), the four consecutive
head maps are
\begin{align}
\BEpoch(n)&:\mathbf b_{\rm in}(n)\longmapsto
                    \mathbf b_{\rm out}(n),\label{eq:Baddressmap}\\
\Bbridge_{n+1}&:\mathbf b_{\rm out}(n)\longmapsto
                    \mathbf a_{\rm in}(n),\label{eq:Bbridgeaddressmap}\\
\AEpoch(n)&:\mathbf a_{\rm in}(n)\longmapsto
                    \mathbf a_{\rm out}(n),\label{eq:Aaddressmap}\\
\Abridge_{n+1}&:\mathbf a_{\rm out}(n)\longmapsto
                    \mathbf b_{\rm in}(n+1).\label{eq:Abridgeaddressmap}
\end{align}
The initial \(\Abridge_1\) maps the finite aligned prefix to
\(\mathbf b_{\rm in}(1)=(36,33)\).  At every splice, the state code,
orientation, absolute block addresses, offsets, and generated target
prefix all agree.
\end{theorem}

\begin{proof}
By Lemma~\ref{lem:bridgeanchors}, the two source words in
\eqref{eq:AEpochSource} begin exactly at the components of
\(\mathbf a_{\rm in}(n)\): the first at \(S_B(n)\), the second at the
first \(6\) of \(\Abridge_n\).  Their terminal occurrences are the last
\(7\) of \(\Bbridge_{n+1}\) and the boundary \(R_B(n)\), which are the
components of \(\mathbf a_{\rm out}(n)\).  Equivalently,
\(\mathbf a_{\rm out}(n)-\mathbf a_{\rm in}(n)=\Delta_A(n)\).
The same argument with the last \(11\) of \(\Abridge_n\) and first
subsequent \(9\) of \(\widehat{\Bbridge}_n\) gives
\(\mathbf b_{\rm out}(n)-\mathbf b_{\rm in}(n)=\Delta_B(n)\).
At \(n=1\), the marked replay of \eqref{eq:B1EpochSource} is explicitly
\[
(36,33)\longmapsto(42,38),
\]
which is the missing B-family base case.

The complete bridge maps, including their domains, are
\begin{align*}
\Abridge_m:(p_A,p_B)&\longmapsto
 \left(p_B+L_B(m),
 p_A+\frac{(m+2)(4m^2+13m+15)}6\right),&&m\ge1,\\
\Bbridge_m:(p_A,p_B)&\longmapsto
 \left(p_B+L_A(m-1),
 p_A+\frac{(m+1)(4m^2+11m+12)}6\right),&&m\ge2.
\end{align*}
They follow by counting the two exact source words in
\eqref{eq:Abridge}--\eqref{eq:Bbridge}.  Substitution of
\eqref{eq:boundaryrec1}--\eqref{eq:Banchors} proves
\eqref{eq:Bbridgeaddressmap} and \eqref{eq:Abridgeaddressmap}
identically in \(n\).  This calculation is independently audited by
\path{code/phase17_symbolic_address_checker.py}.  The
checker represents every expression exactly in the coefficient basis
\(\{4^n,1,n,n^2,n^3,n^4\}\) over \(\mathbb Q\) and reduces \(34\)
factor-length, boundary, head-increment, anchor, and bridge-map residuals
coefficientwise to zero.  It uses no sampled rank or level bound; the
machine-readable residual list is
\path{data/phase17_symbolic_address_residuals.csv}.
Lemma~\ref{lem:alignment} supplies the separate
\(\Abridge_1\) connection to \(\mathbf b_{\rm in}(1)=(36,33)\).

Finally, Lemma~\ref{lem:portfactor} identifies the source factors in their
exact order, while the anchor formulas above identify their global starts.
Lemma~\ref{lem:markedcomposition} therefore lifts every token, epoch, and
bridge to consecutive marked occurrences.  Induction around the four
maps \eqref{eq:Baddressmap}--\eqref{eq:Abridgeaddressmap} shows that each
terminal marked pair is exactly the initial marked pair of the next
factor.  Erasing the marks recovers the unmarked source equations.
\end{proof}

The global factor cycle is shown in Figure~\ref{fig:factorcycle}.

\begin{figure}[htbp]
\centering
\begin{tikzpicture}[
  node distance=18mm and 25mm,
  state/.style={draw,rounded corners,fill=blue!5,text width=32mm,
                minimum height=13mm,align=center},
  arr/.style={-{Latex[length=2.2mm]},thick}
]
\node[state] (ab) {\(\Abridge_n\)\\\(1070714\to17455430\)};
\node[state,right=of ab] (be) {\(\BEpoch(n)\)\\\(17455430\to17373884\)};
\node[state,below=of be] (bb) {\(\Bbridge_{n+1}\)\\\(17373884\to1150227\)};
\node[state,left=of bb] (ae) {\(\AEpoch(n)\)\\\(1150227\to1070714\)};
\draw[arr] (ab)--(be);
\draw[arr] (be)--(bb);
\draw[arr] (bb)--(ae);
\draw[arr] (ae)--(ab);
\node[draw,dashed,fit=(ab)(be)(bb)(ae),inner sep=8mm,
      label=below:{one global level cycle}] {};
\end{tikzpicture}
\caption{The four-factor global state cycle.  Each arrow denotes a complete
all-level source-word transduction.}
\label{fig:factorcycle}
\end{figure}

The initial type-word seed consists of:
\[
(0,0,3),\qquad C_{1,5},\qquad C_{5,2},\qquad C_{2,3}C_{3,8}.
\]
It ends at state \(1070714\), the entry state of \(\Abridge_1\).
Together with the separately replayed boundary identity
\eqref{eq:B1EpochSource}, these are the five finite transduction
obligations used at the base of the global induction.

\begin{theorem}[Global marked-kernel membership]\label{thm:globalmembership}
At the entry of every factor in the infinite construction, the tuple
\[
(z,p_A,\delta_A,p_B,\delta_B,\eps,\text{generated marked prefix})
\]
has the boundary value prescribed by Theorem~\ref{thm:addressgluing}.
Every target block is realized by exactly one of the \(122\) local rules,
and every successor state belongs to the \(92\)-state synchronized domain.
\end{theorem}

\begin{proof}
Induct over the factor sequence
\[
\Abridge_1,\BEpoch(1),\Bbridge_2,\AEpoch(1),
\Abridge_2,\BEpoch(2),\Bbridge_3,\AEpoch(2),\ldots
\]
Lemma~\ref{lem:alignment} proves the complete finite base through
\(\Abridge_1\), including the absolute addresses at block \(74\), and
\eqref{eq:B1EpochSource} supplies the exceptional first B-epoch.
Theorem~\ref{thm:completebridges} handles both bridge families, and
Theorem~\ref{thm:epochsources} handles both epoch families.  Theorem
\ref{thm:addressgluing}, rather than state-code equality alone, shows
that the terminal marked source occurrences, offsets, roles, and target
prefix of one factor are the initial data of the next factor.

Within each factor, every local step is in the 122-rule table and its
successor is in the synchronized domain.  If several table rows have the
same state/current/next key, Proposition~\ref{prop:finite-semantic}
shows that at most one can match both physical source streams; the exact
factor transduction supplies one, hence exactly one.  Induction over the
four-factor cycle therefore preserves the entire displayed tuple and
never leaves the finite domain.
\end{proof}

\section{Causality}\label{sec:causality}

The local rule table defines a finite weighted role graph.  A node records
a synchronized state together with a physical role.  An edge weight is
the change in the lag
\[
j-p_A\quad\text{or}\quad j-p_B
\]
under the corresponding local rule.

The graph has 184 role nodes and 244 edges.  From the entrance to the first
B-epoch \(\BEpoch(1)\), where \(j=74\) and \((p_A,p_B)=(36,33)\), 176 role nodes, 88 cursor states, and 118 rules are reachable.
The remaining four rules occur only in the finite seed.

\begin{proposition}[Lag potential]\label{prop:lag}
There exists a nonnegative integer potential \(\Phi\) on the reachable
role graph such that every edge \(u\to v\) of weight \(w\) satisfies
\[
\Phi(v)\le \Phi(u)+w.
\]
At the entrance to \(\BEpoch(1)\), its values for the physical A- and B-roles are
\(0\) and \(3\), respectively.
There is no reachable negative cycle.  Starting at the entrance to \(\BEpoch(1)\),
\[
j-p_A\ge38,\qquad j-p_B\ge38
\]
for every subsequent local rule.
\end{proposition}

\begin{proof}
The potential values are listed in the v4 certificate \cite{MantovanelliReturnWordSupplementV4}, in
\begin{center}
\path{data/block_lag_potential_certificate.csv}.
\end{center}
The checker verifies the inequality on all \(236\) reachable role edges
(\(118\) recurrent rules, two physical roles per rule).  The eight edges
belonging to the four prefix-only rules are covered by the finite alignment
check.  Bellman--Ford finds no reachable negative cycle.  At the entrance to \(\BEpoch(1)\) the initial lags are
\(j-p_A=74-36=38\) and \(j-p_B=74-33=41\).  Subtract \(38\) from the actual lag and call the resulting
slack \(L\).  Initially \(L=0,3\), matching the two base potentials.  If
an edge of weight \(w\) takes \(u\) to \(v\) and
\(L(u)\ge\Phi(u)\), then
\[
 L(v)=L(u)+w\ge\Phi(u)+w\ge\Phi(v)\ge0.
\]
Induction along every reachable path therefore preserves the lower bound
\(38\).
\end{proof}

\begin{corollary}[Strict backwardness]\label{cor:backward}
Every source bridge used by the global construction lies strictly in the
already generated prefix.
\end{corollary}

\begin{proof}
Theorem~\ref{thm:globalmembership} keeps the infinite run in the reachable
kernel.  Proposition~\ref{prop:lag} gives a positive block lag.  Each
local source bridge spans at most three A-words and four B-words and is
read from the certified source head, hence entirely within the generated
prefix.  The finite part before block \(74\), including
\(\Abridge_1\), is covered directly by Lemma~\ref{lem:alignment}.
\end{proof}

\section{Global semantics and proof of the main theorem}
\label{sec:mainproof}

\begin{theorem}[Global binary solution]\label{thm:globals}
The direct prefix and the infinite return-word construction determine
sequences \(T_n\) and \(s_n\) satisfying the initial data,
\eqref{eq:A} for every \(n\ge2\), and \eqref{eq:B} for every \(n\ge3\), with
\[
s_n\in\{0,1\}
\]
for all \(n\ge2\).
\end{theorem}

\begin{proof}
Lemma~\ref{lem:alignment} supplies the solution directly through the
handoff and identifies its \(C_7\)-coding with the first marked return
words.  The overlapping five-bit windows of adjacent \(C_7\)-codes agree,
so the central bits of the concatenated words define a unique extension
of \(s\).  Set \(T_2=1,T_3=2\), and thereafter define the two parity
subsequences recursively by
\[
T_{n+2}=T_n+2(1-s_n)\qquad(n\ge2).
\]
This defines a unique integer sequence \(T\).  Lemma~\ref{lem:cursorsync}
and the aligned base identify the heads derived from this \(T\) with the
marked kernel cursors at every block boundary and, by the local rule
updates, at every intermediate target position.

By Theorem~\ref{thm:globalmembership}, every future local
configuration belongs to the finite synchronized domain and uses the
correct absolute source occurrences.  Proposition
\ref{prop:finite-semantic} verifies \eqref{eq:B} and excludes the unique
bad configuration \eqref{eq:bad} in that domain.  Lemma
\ref{lem:cursorsync} shows that the same local rules preserve the head
update \eqref{eq:headupdate}.  Corollary
\ref{cor:backward} ensures that every referenced head bit already exists.
Induction over the scalar markers therefore produces a total binary
solution extending the direct prefix.
\end{proof}

\begin{proof}[Proof of Theorem~\ref{thm:main}]
Apply Theorem~\ref{thm:reconstruction} to the total binary solution from
Theorem~\ref{thm:globals}.  The reconstructed \(Q\) satisfies
\eqref{eq:Qrec}, has the prescribed initial values, and all recursive
arguments are positive and less than the current index.
\end{proof}

\section{Local statistics and nonautomaticity of the step sequence}
\label{sec:step-statistics}

We count occurrences in the scalar word of two-step increments
\(s_j=(Q(j+1)-Q(j-1))/2\), with overlapping occurrences counted separately.
For \(N\ge2\) put
\[
 F(N)=\#\{2\le j\le N:s_js_{j+1}\cdots s_{j+5}=001100\}.
\]
The following limits refer to ordinary indices of the original recursion.

\begin{theorem}\label{thm:six-bit-frequency}
Put \(x_m=\lfloor4^m/3\rfloor\) and \(y_m=\lfloor2\cdot4^m/3\rfloor\).
As \(m\to\infty\),
\[
 F(x_m)=\frac{x_m}{96}+O(m^5),\qquad
 F(y_m)=\frac{y_m}{48}+O(m^5).
\]
In particular, the factor \(001100\) has no natural frequency in the
infinite step word.
\end{theorem}

\begin{proof}[Proof of Theorem~\ref{thm:six-bit-frequency}]
For type \(i\), extract the scalar word from Appendix~\ref{app:returnwords}
by replacing each code \(c\) by \(\lfloor c/16\rfloor\bmod 2\).
Every word starts and ends with \(10\). An occurrence of \(001100\)
cannot cross a block boundary. A cut after the first or fifth letter would
put a zero at the start of the next block. At any other interior cut
the four letters around the boundary would be \(1010\), which does not
occur in \(001100\).

The table gives the resulting scalar lengths \(d_i\) and internal counts
\(a_i\). The last two rows give leading multiplicities in the A- and
B-epochs of level \(n\), in units of \(M_A=|A_n|\) and \(M_B=|B_n|\)
respectively. Each entry in these two rows has additive error \(O(n^4)\)
after multiplication by its unit.
\[
\setlength{\arraycolsep}{3.3pt}
\begin{array}{c|rrrrrrrrrrrrr}
 i&0&1&2&3&4&5&6&7&8&9&10&11&12\\ \hline
 d_i&2&31&8&10&39&20&30&40&20&30&31&40&39\\
 a_i&0&0&0&0&1&1&0&0&0&1&2&2&0\\
 A&48&0&16&8&4&8&0&0&0&4&4&4&0\\
 B&48&4&16&8&0&0&4&4&8&0&0&0&4
\end{array}
\]
For clarity, the nonzero offsets of the pattern, indexed from zero, are
\(14\) in types \(4,9\), \(3\) in type \(5\), and \(3,15\) in
types \(10,11\).

To derive the epoch rows, first consider one copy of each of the
four A-tokens. In token \(0\) or \(1\), put \(q=r+2\) for the
assigned rank \(r\). The multiplicities in the two gaps are
\[
\begin{array}{c|c|l}
 \text{gap}&\text{types in order}&\text{multiplicities in that order}\\ \hline
 G_{3A}(r)&10,9,3,0&1,\ q,\ \binom{q+1}{2},\ \binom{q+2}{3}\\
 G_{4A}(r)&4,11,5,2,0&1,\ q,\ \binom{q+1}{2},\
                               \binom{q+2}{3},\ \binom{q+3}{4}
\end{array}
\]
These follow by summing the zero runs of \(P_j\) and the letters of
\(H_j\). The B-gaps have the same multiplicities on the ordered
lists \((12,6,3,0)\) and \((1,7,8,2,0)\), respectively.
For both components the full \(0\)- and \(1\)-tokens each use the
entire rank word once. There are \(M_X+1\) up tokens and \(M_X\)
down tokens in component \(X\). Indeed, the path starts at the
nonzero vertex and ends at zero, and its number of tokens follows
from Lemma~\ref{lem:step-rank-counts}. In component B there are also
two boundary half-loops, contributing one letter apiece.

The rank averages of the four polynomials after the initial \(1\)
in the last row of the gap table are \(3,7,15,31\), up to the
summed errors in Lemma~\ref{lem:step-rank-counts}. In fact
\(\mathbb E[(1+t)^J]=1/(1-t)\), so
\(\mathbb E\binom{J+k+1}{k}=2^{k+1}-1\).
The constant portions of one token of each kind, taken directly
from the finite edge words, are as follows. Commas separate types.
\[
\begin{array}{c|l|l}
 & A & B\\ \hline
 0&(4)&(1)\\
 1&(10)&(12)\\
 U&(4,10)&(12,3,0,6,1,7,2,0,8)\\
 D&(4,5,0,2,11,10,9,0,3)&(12,1)
\end{array}
\]
Adding these entries and the gap counts proves the two epoch rows.
It also shows directly that all B-epochs avoid the pattern, including
the exceptional first epoch.

The dot product with the length row is \(1024\) for each component.
The dot products with the occurrence row are \(32\) for A and \(0\)
for B. Thus the scalar lengths of the level-\(k\) epochs are
\((1024/3)4^k+O(k^4)\) and \((512/3)4^k+O(k^4)\), respectively,
while the A-epoch contributes \((32/3)4^k+O(k^4)\) occurrences.
Lemma~\ref{lem:factorlengths} bounds the total scalar length of all
bridges through level \(n+1\) by \(O(n^5)\), since each block
has at most \(40\) positions.

Let \(e_B(n)\) and \(e_A(n)\) be the last scalar indices of the
two level-\(n\) epochs. At \(e_B(n)\) the completed epochs have
indices \(k\le n\) for B and \(k<n\) for A. At \(e_A(n)\)
both components have completed level \(n\).
Summing the geometric progressions, including all bridge and seed
contributions in the error, gives
\[
\begin{array}{c|cc}
 & e_X(n)&F(e_X(n))\\ \hline
 B& (1024/3)4^n+O(n^5)&(32/9)4^n+O(n^5)\\
 A& (2048/3)4^n+O(n^5)&(128/9)4^n+O(n^5).
\end{array}
\]
Since \(4^5=1024\), the change from \(e_B(n),e_A(n)\) to
\(x_{n+5},y_{n+5}\) moves at most \(O(n^5)\) indices.
The count \(F\) changes by at most one per index, which proves the claim.
\end{proof}

\subsection{Extremal frequencies and the limiting profile}
\label{sec:extremal-frequency}

\begin{theorem}[Frequency profile of \(001100\)]\label{thm:extremal-frequency}
Let \(z_m=4^m/3\) and define
\[
 \Phi_{001100}(t)=
 \begin{cases}
 \displaystyle\frac1{32}-\frac1{48t},&1\le t\le2,\\[2mm]
 \displaystyle\frac1{24t},&2\le t\le4.
 \end{cases}
\]
Then
\begin{equation}\label{eq:extremal-profile}
 \lim_{m\to\infty}\ \sup_{1\le t\le4}
 \left|\frac{F(\lfloor tz_m\rfloor)}{\lfloor tz_m\rfloor}
       -\Phi_{001100}(t)\right|=0.
\end{equation}
In particular,
\[
 \liminf_{N\to\infty}\frac{F(N)}N=\frac1{96},
 \qquad
 \limsup_{N\to\infty}\frac{F(N)}N=\frac1{48}.
\]
The set of all subsequential limits is the closed interval
\([1/96,1/48]\).
\end{theorem}

The endpoint estimates in Theorem~\ref{thm:six-bit-frequency} locate the
two ends of the A-epoch on this scale. To control the intervening scalar
cuts we first estimate sums over arbitrary prefixes of the rank words.
In the notation of Lemma~\ref{lem:step-rank-counts}, write
\[
 R_0=(0),\qquad R_h=D_h\,\sigma_h(R_{h-1})\quad(h\ge1),
 \qquad D_h=(h,h-2,\ldots,h\bmod2).
\]
Here \(\sigma_h=U\) for even \(h\) and \(\sigma_h=V\) for odd \(h\).
Both substitutions replace \(r\) by a permutation of \(0,\ldots,r\).
As before, \(J\) has distribution \(\mu_j=2^{-j-1}\), \(j\ge0\).

\begin{lemma}[Uniform rank-prefix estimate]\label{lem:uniform-rank-prefix}
For every polynomially bounded function \(p\) on the nonnegative integers,
\[
 \sup_{P\text{ prefix of }R_h}
 \left|\sum_{r\in P}p(r)-|P|\,\mathbb E p(J)\right|=o(2^h).
\]
The sum counts entries with multiplicity. The estimate also holds for
suffixes.
\end{lemma}
\begin{proof}
Set \(f=p-\mathbb E p(J)\) and \((Tf)(r)=\sum_{j=0}^r f(j)\).
Choose \(C,d\) so that \(|f(r)|\le C(1+r)^d\). Fix an integer
\(k\ge1\) and take \(h\ge k\). Expanding the last \(k\) stages of the
grammar writes \(R_h\) as a boundary prefix followed by the \(k\)-fold
images of the entries of \(R_{h-k}\). Each such image of \(r\) has
\(f\)-sum \(T^kf(r)\) and length \(T^k1(r)=\binom{r+k}{k}\):
each substitution sums over \(0,\ldots,r\), regardless of their order.
The boundary prefix consists of \(D_h\) and the images of
\(D_{h-i}\) under \(i\) substitutions, \(1\le i<k\). Before substitution
each \(D_{h-i}\) has at most \(h+1\) letters of rank at most \(h\). For fixed \(k\),
the total absolute contribution from these images and one partially
read \(k\)-fold image is \(O_{f,k}((h+1)^{d+k+1})\). Hence
\[
 \sup_{P\text{ prefix of }R_h}\left|\sum_{r\in P}f(r)\right|
 \le \sum_{r=0}^{h-k}c_{h-k}(r)|T^kf(r)|
       +O_{f,k}((h+1)^{d+k+1}).
\]
By Lemma~\ref{lem:step-rank-counts}, \(c_j(r)/|R_j|\to\mu_r\)
for fixed \(r\), with \(c_j(r)/|R_j|\le C'2^{-r}\) after extension
by zero for \(r>j\). Since \(T^k f\) is polynomially bounded for
fixed \(k\), dominated convergence and \(|R_{h-k}|/|R_h|\to2^{-k}\)
give
\begin{equation}\label{eq:fixed-depth-prefix}
 \limsup_{h\to\infty}\frac{1}{|R_h|}
 \sup_{P\text{ prefix of }R_h}\left|\sum_{r\in P}f(r)\right|
 \le 2^{-k}\mathbb E|T^kf(J)|.
\end{equation}

Set \(a_j=\mu_jf(j)\) and \(b_k=\mu^{*k}\), where the sequences
are extended by zero to negative indices and \(*\) denotes convolution.
Because \(f\) is centered, \(\sum_j a_j=0\). Direct convolution gives
\[
 \mu_r2^{-k}T^kf(r)=(a*b_k)(r),\qquad
 b_k(r)=\binom{r+k-1}{k-1}2^{-r-k}\quad(r\ge0).
\]
The sequence \(b_k\) is unimodal. Indeed its ratio at consecutive
indices \(r,r+1\) is \(b_k(r+1)/b_k(r)=(r+k)/(2r+2)\).
For translation \((\tau_j b)(r)=b(r-j)\), summing the absolute
first differences of a unimodal probability mass function yields
\[
 \|\tau_j b_k-b_k\|_1\le 2j\max_r b_k(r),\qquad
 \max_r b_k(r)=\binom{2k-2}{k-1}2^{-(2k-1)}=O(k^{-1/2}).
\]
The last bound follows from Stirling's formula. Now
\(a*b_k=\sum_{j\ge0}a_j(\tau_j b_k-b_k)\) and
\(\sum_jj|a_j|<\infty\), so
\[
 2^{-k}\mathbb E|T^kf(J)|=\|a*b_k\|_1
 \le 2\max_r b_k(r)\sum_{j\ge0}j|a_j|\longrightarrow0.
\]
Letting \(k\to\infty\) in \eqref{eq:fixed-depth-prefix} proves the
prefix estimate. Subtracting a prefix sum from the whole-word sum
gives the suffix estimate.
\end{proof}

\begin{lemma}[Prefixes of an A-epoch]\label{lem:scalar-epoch-prefix}
Let \(H_n(L)\) count the occurrences of \(001100\) whose starts lie in
the first \(L\) scalar positions of \(\AEpoch(n)\). Uniformly for
\(0\le L\le |\AEpoch(n)|_{\rm scalar}\),
\[
 H_n(L)=L/32+o(4^n)\qquad(n\to\infty).
\]
\end{lemma}
\begin{proof}
The token path starts with \(0^nU\). Subsequently each positive
rank \(r\) of \(B_{n+1}\) produces \(D0^{r-1}U\), and each zero
produces a \(1\)-token, by Section~\ref{sec:rankgrammar}. Each group has
\(O(n)\) tokens since the largest rank of \(B_{n+1}\) is \(2n-1\).
Up to the initial \(O(n)\) tokens and one partial group, the four token
counts \(q_0,q_1,q_U,q_D\) in any token prefix are sums over a common
prefix of \(B_{n+1}=R_{2n-1}\) of the respective functions
\[
 (r-1)_+,\qquad \mathbf 1_{\{r=0\}},\qquad
 \mathbf 1_{\{r>0\}},\qquad \mathbf 1_{\{r>0\}}.
\]
All four expectations under \(J\) are \(1/2\). Thus
Lemma~\ref{lem:uniform-rank-prefix} gives
\(q_i=q_0+o(4^n)\) uniformly over token prefixes.

The ranks assigned to the \(0\)-tokens form a prefix of \(A_n\).
Those assigned to the \(1\)-tokens form a prefix of its reversal.
For a single A-token the exact scalar length \(\ell\) and pattern
count \(a\) follow by weighting its type multiplicities in the proof
of Theorem~\ref{thm:six-bit-frequency} by \(d_i\) and \(a_i\).
The expectation columns use \(r=J\).
\[
 \begin{array}{c|cccc}
 \text{token}&\ell&a&\mathbb E[\ell]&\mathbb E[a]\\ \hline
 0&2\binom{r+9}{3}&r+5&260&6\\
 1&2\binom{r+9}{4}&\binom{r+5}{2}&512&16\\
 U&70&3&70&3\\
 D&182&7&182&7
 \end{array}
\]
For example, the mean length of token \(0\) is twice the coefficient
of \(u^3\) in \((1+u)^9/(1-u)\), since
\(\mathbb E(1+u)^J=(1-u)^{-1}\). Each occurrence belongs to a
single return word by Theorem~\ref{thm:six-bit-frequency} and its proof.
Apply Lemma~\ref{lem:uniform-rank-prefix} to the length and count
polynomials, using suffixes for the reversed queue. For scalar length
\(L\) and count \(H\) of a complete token prefix this yields, uniformly,
\begin{align*}
 L&=260q_0+512q_1+70q_U+182q_D+o(4^n)=1024q_0+o(4^n),\\
 H&=6q_0+16q_1+3q_U+7q_D+o(4^n)=32q_0+o(4^n).
\end{align*}
Thus \(H-L/32=o(4^n)\). Every assigned rank is at most \(2n-2\),
so one token has scalar length \(O(n^4)\). Moving a scalar cut back
to the previous token boundary therefore changes both length and
count by \(O(n^4)\). Counting starts through a cut can add at most
five occurrences to the number fully contained in the prefix.
\end{proof}

\begin{proof}[Proof of Theorem~\ref{thm:extremal-frequency}]
Let \(e_B(n),e_A(n)\) be the last scalar indices of the level-\(n\)
epochs. With \(m=n+5\), the endpoint calculation in the proof of
Theorem~\ref{thm:six-bit-frequency} gives
\begin{align*}
 e_B(n)&=z_m+O(m^5),& F(e_B(n))&=z_m/96+O(m^5),\\
 e_A(n)&=2z_m+O(m^5),& F(e_A(n))&=z_m/24+O(m^5).
\end{align*}
Between \(e_B(n)\) and \(e_A(n)\) come the bridge \(\Bbridge_{n+1}\)
and the A-epoch of level \(n\). Between \(e_A(n)\) and \(e_B(n+1)\)
come \(\Abridge_{n+1}\) and the next B-epoch. Lemma~\ref{lem:factorlengths}
bounds the total scalar length of the bridges through these levels by
\(O(m^5)\), since a block has length at most \(40\).
Use Lemma~\ref{lem:scalar-epoch-prefix} in the first interval and the
absence of the pattern from B-epochs in the second. The bound
\(0\le F(N+1)-F(N)\le1\) absorbs the bridge lengths and the endpoint
displacements. It follows uniformly on the indicated intervals that
\[
 F(N)=
 \begin{cases}
 N/32-z_m/48+o(4^m),&z_m\le N\le 2z_m,\\
 z_m/24+o(4^m),&2z_m\le N\le4z_m.
 \end{cases}
\]
Rounding an endpoint changes the count by at most one. Substituting
\(N=\lfloor t z_m\rfloor\) and dividing by \(N\) proves
\eqref{eq:extremal-profile}, since \(N\ge z_m-1\).

The continuous profile increases from \(1/96\) to \(1/48\) on
\([1,2]\) and decreases back to \(1/96\) on \([2,4]\). Every sufficiently
large integer lies in an interval \([z_m,4z_m]\). Thus every subsequential
limit belongs to the range of the profile. Conversely each fixed phase
\(t\in[1,4]\) realizes \(\Phi_{001100}(t)\) along
\(N_m=\lfloor t z_m\rfloor\), proving the claimed interval and its extrema.
\end{proof}

At the two distinguished phases \(t=1,2\),
Theorem~\ref{thm:six-bit-frequency} gives a frequency error of
\(O(m^5/4^m)\). Equation~\eqref{eq:extremal-profile} gives the
uniform error \(o(1)\) over the full phase interval.

\subsection{Alternating factors and automaticity}

Write \(L(N)\) for the maximum length of an alternating binary factor
of \(s_2\cdots s_N\). A sequence is \(k\)-automatic if a finite automaton
with output computes its \(n\)-th term from the base-\(k\) digits of \(n\).

\begin{theorem}\label{thm:step-nonautomatic}
As \(N\to\infty\),
\[
 L(N)=2\log_2 N+O(1).
\]
For every integer \(k\ge2\), both the binary step word \(s\) and
the residue sequence \((Q(n)\bmod 4)_{n\ge1}\) are nonautomatic
in base \(k\).
\end{theorem}

\begin{proof}
We first count consecutive copies of \emph{type} \(0\) in the
return-word grammar. Let \(z(w)\) be the largest exponent of a run of
this type in a finite target word \(w\), with \(z=0\) for an absent
type. For \(m\ge1\), \(P_m\) and \(Q_m\) each have maximal zero run \(m\),
because \(\pi_m\) contains every integer through \(m\) exactly once.
The same maximum holds for \(H_m,K_m\), whose constituent words are
separated by nonzero types. Consequently each of the four gaps of
rank \(r\ge0\) has maximum \(r+2\). The finite connector words of the
four token types displayed above have maximum at most one and begin
and end in nonzero types. Each gap also begins and ends in a nonzero
type. Thus the maximum over each epoch is read off from its largest
rank. For \(A_n\) and \(B_n\) these ranks are \(2n-2\) and \(2n-3\),
respectively. At \(n=1\), the B-epoch has no full ranked token.
Its two half-loops and its one up token give maximum one directly.

In \(\Abridge_n\), the words outside the two tails in
\eqref{eq:Abridgeword} have zero runs at most \(2n\).
The tail \(\EA(n)\) gives the run \(2n+3\) and \(\DA(n)\) has
maximum \(2n+1\). Their join is separated by type \(3\), and the
outer joins are separated by types \(5,6\). In \(\Bbridge_m\),
\(m\ge2\), the words outside the tails have runs at most \(2m-1\).
The tail \(\EB(m)\) supplies \(2m+2\), while \(\DB(m)\) has
maximum \(2m\). The constituent factors on either side of these tails
and the join of the tails have nonzero separators in
\eqref{eq:Bbridgeword}. The four consecutive maxima at every
level \(n\ge1\) are therefore
\[
\begin{array}{c|cccc}
 w&\Abridge_n&\BEpoch(n)&\Bbridge_{n+1}&\AEpoch(n)\\ \hline
 z(w)&2n+3&2n-1&2n+4&2n .
\end{array}
\]
These factors all begin and end in nonzero types. The initial finite
prefix affects the running maximum by at most a constant.

The scalar word for type \(0\) is \(10\). Every other type has an
internal equal adjacent pair, by the same decoding of
Appendix~\ref{app:returnwords}. A scalar alternating factor can
therefore extend beyond a run of type \(0\) by at most two partial
nonzero words, of total length at most \(80\). Suppose position \(N\)
lies in level \(n\), starting at \(\Abridge_n\) and ending at
\(\AEpoch(n)\). For \(n\ge2\) the complete preceding level gives a
run of \(2n+2\) zero types. The current level has no run longer than
\(2n+4\). Hence \(L(N)=4n+O(1)\). The factor boundaries of
Lemma~\ref{lem:factorlengths} grow as \(\Theta(4^n)\), and passing
from blocks to scalar positions multiplies lengths by factors
between \(2\) and \(40\). Thus \(\log_2 N=2n+O(1)\).

Let \(d_1<d_2<\cdots\) be the indices with \(s_{d_i}=s_{d_i+1}\).
There are infinitely many, since every A-epoch contains a nonzero
word. After the finite initial alternating run, the maximal alternating
runs are precisely \(s_{d_i+1}\cdots s_{d_{i+1}}\). Since \(L(N)\)
is unbounded, their lengths \(d_{i+1}-d_i\) are unbounded. Each such
run lies in the prefix ending at \(d_{i+1}\), so for all large \(i\)
and a fixed \(C>0\),
\[
 d_{i+1}-d_i\le L(d_{i+1})\le C\log d_{i+1},\qquad
 0\le 1-\frac{d_i}{d_{i+1}}
    \le \frac{C\log d_{i+1}}{d_{i+1}}\longrightarrow0.
\]
Thus \(d_i/d_{i+1}\to1\), and taking reciprocals gives
\(d_{i+1}/d_i\to1\). By Cobham's gap theorem
\cite[Theorem~10, p.~184]{Cobham1972}, the increasing enumeration
\((x_i)\) of any infinite \(k\)-recognizable set has bounded successive
differences or satisfies \(\limsup_i x_{i+1}/x_i>1\).
The sequence \((d_i)\) violates both alternatives, so the set of
equal-pair indices is unrecognizable in every integer base \(k\ge2\).
If \(s\) were \(k\)-automatic, so would be the indicator of these pairs
by closure under shifts and pointwise maps. Finally,
\(s_j=1\) exactly when \(Q(j-1)\) and \(Q(j+1)\) have different
residues modulo \(4\), by \eqref{eq:Qdiff}. The same closure properties
therefore rule out \(k\)-automaticity of \(Q\bmod4\).
\end{proof}

\section{Computational certificate and trusted base}
\label{sec:verification}

The v4 archive \cite{MantovanelliReturnWordSupplementV4} contains the
complete finite kernel, the word generators and the induction certificates.
The kernel audit checks the finite obligations in Table~\ref{tab:finitechecks}.
The additional checker \path{anc/check_v4.py} verifies gap entries and exits,
physical lookaheads, the two half-loops and the active predecessor condition.
It also checks ordered projections at finite levels, scalar token lengths
and pattern counts at sample ranks, their expectations under the geometric
rank distribution, and selected six-letter counts from the original recursion.
The proof uses finite verification only where the statement itself is
finite.

\begin{table}[htbp]
\centering
\small
\begin{tabularx}{\textwidth}{Y r r}
\toprule
certificate component & objects checked & violations\\
\midrule
stationary primitive paths & 25 & 0\\
stationary loop schemas & 4 & 0\\
tail and bridge connector primitives & 15 & 0\\
tail loop schemas & 6 & 0\\
synchronized states & 92 & 0\\
local transition rules & 122 & 0\\
semantic A/B windows & 28\,025 & 0\\
reachable role-graph edges & 236 & 0\\
finite seed transductions & 5 & 0\\
finite recurrence-to-kernel alignment & 71 blocks & 0\\
derived cursor transforms & 13 & 0\\
ambiguous selector rule pairs & 36 & 0\\
exact symbolic address identities & 34 & 0\\
\bottomrule
\end{tabularx}
\caption{Finite proof obligations checked exhaustively.}
\label{tab:finitechecks}
\end{table}

Diagnostic expansions additionally check:
\begin{itemize}
\item 232 stationary chunks through parameter \(60\);
\item 398 linear tails through level \(100\);
\item 39 complete bridges through level \(20\);
\item 349\,504 epoch macrosegments and 87\,364 gap anchors through level \(8\);
\item 443 exact source identities, totaling 10\,919\,761 local target steps.
\end{itemize}
These cutoffs are regression tests, not logical hypotheses.

To run both parts of the verification, execute the following command
in the directory containing \path{SHA256SUMS.txt}.
\begin{verbatim}
python reproduce.py
\end{verbatim}

\subsection{Trusted base}

The trusted base consists of:
\begin{enumerate}
\item the mathematical reductions and gluing arguments printed in this
paper;
\item the machine-readable 122-rule table and finite kernel data;
\item the short deterministic checkers;
\item the Python runtime and standard library used to execute them.
\end{enumerate}
The result is not claimed to be formalized in Lean, Coq, Isabelle, or HOL.
An independent reimplementation of the finite replay and loop checks would
further reduce implementation risk.

\section{Conclusion}
The ordered return-word construction identifies each scalar step of
the parity-perturbed recursion. Within an A-epoch the count of
\(001100\) grows at rate \(1/32\), up to a uniform error of smaller
order than the epoch length. Within a B-epoch the count is constant.
After rescaling, these two phases give the frequency profile of
Theorem~\ref{thm:extremal-frequency}, whose range \([1/96,1/48]\)
is exactly the set of subsequential limits of \(F(N)/N\).
The logarithmic growth of the longest alternating factor rules out
automaticity for the step sequence and for \(Q\bmod4\).

\section*{Data and code availability}
The complete machine-readable certificate for this manuscript is packaged
in a Zenodo archive \cite{MantovanelliReturnWordSupplementV4},
DOI \href{https://doi.org/10.5281/zenodo.24441831}{\texttt{10.5281/zenodo.24441831}}.
It contains the finite kernel, transition tables, source words, reference
audits, verification programs, and the manuscript in PDF and \LaTeX{} form.
After extraction, run \code{python reproduce.py} from the archive's
top-level directory with Python~3.11 or newer. The program checks the
file-level SHA-256 hashes and runs both the kernel audit and the additional
v4 checks offline. It prints the directory containing the new results.

\section*{Acknowledgements}

The author is grateful to Beno\^{\i}t Cloitre for his detailed analysis
of the sequence and for comments on earlier versions of the manuscript
that helped clarify the limitations of the previous finite-state
formulation.

\medskip
\noindent\textbf{AI disclosure.}
The author used OpenAI’s ChatGPT for help in language editing,
LaTeX restructuring, mathematical explorations, and the preparation of verification
scripts. All mathematical statements, proofs, computations, references, and the final
presentation were independently verified by the author, who assumes full responsibility
for the content.

\appendix

\section{The 13 return words}\label{app:returnwords}

The following table reproduces the machine-readable return-word list.
Spaces separate \(C_7\)-codes.

\begingroup
\footnotesize
\setlength{\tabcolsep}{4pt}
\begin{longtable}{C{10mm} C{12mm} >{\raggedright\arraybackslash}p{0.72\textwidth}}
\toprule
type & length & \(C_7\)-word\\
\midrule
\endfirsthead
\toprule
type & length & \(C_7\)-word\\
\midrule
\endhead
0&2&\code{84 43}\\
1&31&\code{84 43 84 47 89 49 100 73 20 43 84 47 89 49 101 77 27 52 106 80 36 73 20 47 89 49 101 77 27 52 107}\\
2&8&\code{84 43 80 37 77 27 52 107}\\
3&10&\code{84 43 84 47 89 49 100 73 20 43}\\
4&39&\code{84 43 84 47 89 49 100 73 22 46 91 52 107 80 37 77 25 49 100 73 20 43 84 47 89 49 100 73 22 46 91 52 107 80 37 77 27 52 107}\\
5&20&\code{84 43 80 37 77 25 49 100 73 22 46 91 52 107 80 37 77 27 52 107}\\
6&30&\code{84 43 84 47 89 49 100 73 20 43 84 47 89 49 101 77 27 52 106 80 36 73 20 47 89 49 100 73 20 43}\\
7&40&\code{84 43 80 37 77 27 52 106 80 36 73 20 47 89 49 100 73 20 43 84 47 89 49 101 77 27 52 106 80 36 73 20 47 89 49 101 77 27 52 107}\\
8&20&\code{84 43 80 37 77 27 52 106 80 36 73 20 47 89 49 101 77 27 52 107}\\
9&30&\code{84 43 84 47 89 49 100 73 22 46 91 52 107 80 37 77 25 49 100 73 20 43 84 47 89 49 100 73 20 43}\\
10&31&\code{84 43 80 37 77 25 49 100 73 22 46 91 52 107 80 37 77 25 49 100 73 20 43 84 47 89 49 100 73 20 43}\\
11&40&\code{84 43 80 37 77 25 49 100 73 22 46 91 52 107 80 37 77 25 49 100 73 20 43 84 47 89 49 100 73 22 46 91 52 107 80 37 77 27 52 107}\\
12&39&\code{84 43 80 37 77 27 52 106 80 36 73 20 47 89 49 100 73 20 43 84 47 89 49 101 77 27 52 106 80 36 73 20 47 89 49 100 73 20 43}\\
\bottomrule
\end{longtable}
\endgroup

\section{Finite connector inventory}

The complete connector words and source bridges are contained in the Zenodo archive.  Table~\ref{tab:keyconnectors} lists the seven bridge
connectors used in Theorem~\ref{thm:completebridges}.

\begin{table}[htbp]
\centering
\small
\begin{tabularx}{\textwidth}{l c c Y}
\toprule
connector & start & end & target factor\\
\midrule
A-prefix & 1070714 & 17435506 & \(4H_1\)\\
A \(11/P_2\) & 17435506 & 17337218 & \(11\,5\phi_2(P_2)\)\\
A \(5\) & 17337218 & 16778021 & \(5\)\\
A-suffix & 17352080 & 17455430 & \(6\phi_3(Q_2)6\)\\
B-prefix & 17373884 & 610744 & \(1\,7K_2\)\\
B \(7/Q_1\) & 610744 & 4440 & \(7\phi_2(Q_1)8\)\\
B-suffix & 1101155 & 1150227 & \(9\phi_3(P_1)\)\\
\bottomrule
\end{tabularx}
\caption{Finite bridge connectors.}
\label{tab:keyconnectors}
\end{table}

\section{Ordered source projections at every level}
\label{app:portrecursions}
\subsection{Rank runs and the two A-epoch projections}\label{app:A-projection-proof}

For a rank word $w$, let $\overline w$ denote reversal and let $F(w)$
delete every zero and subtract one from each positive letter. Set
$B_1=A_0=O_0=\varnothing$ in this appendix; these are empty rank
queues, not target epochs. For a word $p=(p_1,\ldots,p_t)$, put
\[
 {\cal I}_h(p)=
 \left(\prod_{i=1}^t 0^{p_{t+1-i}}(p_i+1)\right)0^h .
\]
Products are ordinary word concatenations, and an empty product is the
empty word.

\begin{lemma}[Ordered zero runs]\label{lem:port-rank-skeleton}
For every $n\ge1$,
\[
 A_n={\cal I}_n(B_n),\qquad B_{n+1}={\cal I}_n(A_n).
\]
In particular $F(A_n)=B_n$ and $F(B_{n+1})=A_n$.
\end{lemma}
\begin{proof}
The parity orders give $F(U(r))=V(r-1)$ and $F(V(r))=U(r-1)$
when $r>0$, and both images are empty at zero. Since
$F(E_n)=O_{n-1}$ and $F(O_n)=E_n$, alternating induction in the
mutual rank factorizations gives the two asserted $F$ identities,
starting with $A_1=0$, $B_2=(1,0)$.

It remains to locate the zeros. Write $p=(p_1,\ldots,p_t)$ and
$g_i=p_{t+1-i}$. Suppose $B_n={\cal I}_{n-1}(p)$ with
$p=A_{n-1}$. In $E_nU(B_n)$ the first $n-1$ positive letters
have no preceding zero. The run just before the first positive letter
of $U(p_i+1)$ has length $g_i+1$: one zero from the preceding factor
(or from $E_n$) and $g_i$ letters $U(0)$. The other $p_i$ positive
letters of this factor have zero preceding run length. The final
run has length $1+(n-1)=n$. The list of run lengths before successive
positive letters is consequently
\[
 0^{n-1}\prod_{i=1}^t (g_i+1)0^{p_i}
 =\overline{B_n}.
\]
Together with $F(A_n)=B_n$ this proves $A_n={\cal I}_n(B_n)$.
For the other step write $A_n={\cal I}_n(p)$ with $p=B_n$.
In $O_nV(A_n)$ the initial $n$ positives have zero preceding
run length. Every $V(p_i+1)$ begins with one zero preceded by
$g_i$ images $V(0)=0$, and has $p_i$ further positive letters.
Its trailing run, after the last positive factor, consists of the
$n$ images $V(0)$. Thus the run-length list is
$0^n\prod_i(g_i+1)0^{p_i}=\overline{A_n}$, and
$F(B_{n+1})=A_n$ proves the second formula. The initial two words
satisfy the claimed descriptions directly.
\end{proof}

\paragraph{Common notation.}
Put $D_m=\phi_3(Q_m)6$ and $J_m=7K_m$, and write
\[
 X=(12,1),\quad
 L(r)=12\left(\prod_{m\in\pi_{r+2}}D_m\right)1,\qquad
 H(r)=X\prod_{m\in\rho_{r+2}}J_m \qquad(r\ge-1).
\]
Here $D_1=(3,0,6)$ and $J_1=(7,2,0,8)$. The full B tokens
$0,U,D,1$ have target words $L(r),L(-1)J_1,X,H(r)$,
respectively. The initial and terminal B half-loops contribute $1$
and $12$. When spelling a physical A-epoch source, keep its first
letter once and append each source token with its first letter
deleted. This implements the inclusive overlap exactly.

Fix $n\ge1$ and write $Y=A_n$ and $Z=B_{n+1}=O_nV(Y)$.
The path on $A_{n+1}=E_{n+1}U(Z)$ begins with $n$ tokens $0$ and
one $U$. Each positive letter $x$ of $Z$ then gives
$D\,0^{x-1}U$, and every zero of $Z$ gives a token $1$. In order,
the positives of $Z$ are $x=y+1$ for $y$ in $Y$. After the
initial $n$ zero-tokens have consumed $E_n$, the forward rank queue
is $U(B_n)=U(F(Y))$. Thus the $x-1=y$ zero-tokens belonging to
$x=y+1$ receive $\pi_{y-1}0$ if $y>0$, and receive no ranks if
$y=0$.

All positive runs of $Z$ end in a single $1$, with all other letters
in that run greater than $1$. This follows from $O_nV(Y)$, since
every nonempty $\rho_y$ ends in $1$. The run of zeros following
successive occurrences of $1$ in $Z$ has length $s+1$ for
$s$ running through $\overline{B_n}$, followed by one terminal
run of length $n$. Indeed the zeros preceding the nonempty
$\rho_y$ have the lengths of the zero runs before positive letters
of $Y$, each increased by one for the initial zero of $V(y)$.
Lemma~\ref{lem:port-rank-skeleton} gives these lengths and the
terminal run of $Y$. During a nonterminal run of $s+1$ zeros the
reverse A-rank queue is therefore exactly $0\rho_s$; during the
terminal run it is $\overline{E_n}=(0,2,\ldots,2n-2)$. This
follows by splitting $\overline{Y}=\overline{U(B_n)}\,\overline{E_n}$
into the groups $0\rho_s$ in reverse $B_n$ order.

\paragraph{First physical projection.}
Stripping the initial letters of the A sources, the contributions of
$0,U,D,1$ are respectively
\[
  L(r-1),\quad 7,\quad K_1X,\quad K_{r+2}7.
\]
The physical port starts at $1$. The initial $n$ zero-tokens
give $\prod_{p\in E_n}L(p-1)$. Its leading $n-1$ terms are the
full B-zero tokens with forward ranks $O_{n-1}$. Its last term is
$L(-1)$. The following $U$ and the $K_1$ from the first $D$
complete the direction token $L(-1)J_1$. These are exactly the
B target contributions up to the first zero of $Z$, since $O_n$
is an initial run of $n$ positives with the first half-loop.

To compare the remaining words, place a cut immediately before the
$X$ in the $D$ attached to each positive $x$ of $Z$. On the B
target place the corresponding cut at the token immediately after
the zero at the start of $V(y)$, where $x=y+1$. This next token
always starts with $X$: it is $D$, a full $1$, or the terminal
half-loop $12$ followed by the first $1$ of the bridge prefix.
There is one cut per letter $y$ in matching order.
\begin{itemize}
\item If $y>0$, then $x>1$ and is immediately followed by another
positive letter in $Z$. The interval from this cut up to the next
cut on the source is
\[
 X\left(\prod_{p\in\pi_{y-1}}L(p-1)\right)L(-1)J_1 .
\]
On the target, $\rho_y$ is a run of $y$ positives, giving a $D$,
then $y-1$ full B-zero tokens, then the $U$ into the next zero.
Their forward ranks are exactly $\pi_{y-1}-1$: for $y\ge2$ this
equals $V(y-2)$ and for $y=1$ it is empty. Across all these cells
the groups concatenate to $V(F^2(Y))=V(A_{n-1})$ for $n\ge2$,
the entire remaining forward B queue after $O_{n-1}$. At $n=1$
there are no positive cells.
\item If $y=0$ and this is not the last cell, $x=1$ is followed
by a run of $s+1$ zeros, with reverse ranks $0\rho_s$ as above.
The initial $U$ supplies $7$, the successive $1$-tokens supply
$K_{p+2}7$, and the next $D$ supplies its initial $K_1$.
The cell is consequently
\[
 X\left(\prod_{p\in0\rho_s}J_{p+2}\right)J_1=H(s),
\]
because $\rho_{s+2}=2(\rho_s+2)1$. On the target $V(0)$ is
followed by another zero; the intervening full B-one token uses
$s$. The values $s$ exhaust its reverse queue $\overline{B_n}$
in order.
\item The terminal cell has $y=0$. The $n$ zeros after its
$x=1$ consume $\overline{E_n}$, and no $D$ follows. It spells
$XJ_2J_4\cdots J_{2n}7$, the terminal target half-loop $12$
followed by the bridge prefix $1J_2J_4\cdots J_{2n}7$.
\end{itemize}
This proves the first ordered projection in \eqref{eq:AportA}.

\paragraph{Second physical projection.}
Here the initial letter is $6$. The contributions of $0,U,D,1$
after removal of their initial letters are respectively
\[
  D_{r+2},\quad 1,\quad 12D_1,\quad
     \left(\prod_{m\in\rho_{r+1}}J_m\right)X .
\]
The last identity uses $G_{4B}(r-1)=1\prod_{m\in\rho_{r+1}}J_m$;
at $r=0$ it uses $C_{5,7}=1J_1$. The $n$ initial zero-tokens
and the initial $U$ yield
$6D_{2n}D_{2n-2}\cdots D_2\,1$. Its part before $1$ is exactly
the bridge suffix $\suf_6(\Abridge_n)$, while $1$ is the initial
B half-loop.

For the positive $x=y+1$ of $Z$, the three parts $D\,0^yU$
give $L(y-1)$. Indeed, for $y>0$ the $D_m$ subscripts are
$1,(\pi_{y-1}+2),2$, in precisely the order $\pi_{y+1}$; for
$y=0$ only $D_1$ occurs. If $y>0$, this positive $x$ is followed
immediately by a positive in $Z$ and $L(y-1)$ is the B-zero
token for the transition to that positive. The ranks $y-1$ in
these transitions read $F(Y)=B_n$ forward. If $y=0$, then
$x=1$ ends the positive run and gives $L(-1)$. The first following
zero has reverse rank $0$ and contributes $J_1 X$. Its $J_1$
completes the B direction token $L(-1)J_1$ entering that zero.
The remaining $X$ is the start of the token leaving that zero.

In a nonterminal zero run the remaining ranks $p$ form $\rho_s$;
in the terminal zero run they are $2,4,\ldots,2n-2$. Each such
additional zero encodes a B-one token, since its preceding $X$
and its own initial product of $J_m$ spell $H(p-1)$ and leave the
next $X$. The B-one ranks obtained over the complete word are
\[
 \left(\prod_{\substack{s\in\overline{B_n}\\ s>0}}U(s-1)\right)
 \overline{O_{n-1}}
 =\overline{V(A_{n-1})}\,\overline{O_{n-1}}
 =\overline{B_n}\qquad(n\ge2).
\]
Indeed $\rho_s-1=U(s-1)$ for $s>0$, and deleting the zero letters
of $\overline{B_n}$ and subtracting one gives
$\overline{A_{n-1}}$. At $n=1$ both queues are empty. The last
$X$ of a nonterminal zero run supplies the B direction token $D$
to the next positive. The last $X=(12,1)$ of the terminal run
supplies the terminal B half-loop and the final source letter $1$.
This proves \eqref{eq:AportB}, including its orientation and both
endpoints.

\subsection{The two B-epoch projections}\label{app:B-projection-proof}

Fix $n\ge2$ and put $X=A_{n-1}$, $Y=B_n$, $Z=A_n$. The rank
factorizations and the map $F$ of the preceding subsection give
\[
 Y=O_{n-1}V(X),\qquad Z=E_nU(Y),\qquad F(Y)=X,\quad F(Z)=Y.
\]
Consequently the positive entries of $Y$ are $x+1$, for $x$ running
through $X$ in order. Those of $Z$ are $y+1$, for $y$ running through
$Y$ in order. Both $X$ and $Y$ end in zero. All the identities below
use ordinary concatenation. In particular, a source token supplies
its word with the shared first letter deleted, after the initial
letter of the entire source word has been written once.

Write $o,i$ for a full zero- or one-token. The target of the A-epoch
at level $n-1$ is read from $Z=E_nU(Y)$. Its token path starts with
$o^{n-1}U$, followed, for each successive $y\in Y$, by
\[
 D o^{y-1}U\quad(y>0),\qquad i\quad(y=0).
 \tag{C.1}\label{eq:port-b-a-path}
\]
Indeed the initial segment ends at the zero terminating $E_n$.
A nonempty $\pi_y$ opens a positive run of length $y$ and the zero
at the end of $U(y)$ closes it. The A forward and reverse rank queues
are $X$ and $\rev(X)$.

The B-epoch at level $n$ is read from $B_{n+1}=O_nV(Z)$ followed by
the terminal half-loop. Its full-token path, between the two half-loops,
also starts with $o^{n-1}U$. Next, each positive entry $z$ of $Z$
gives $D o^{z-1}U$, and each zero of $Z$ except its last gives $i$.
The first zero of $V(Z)$ supplies the initial $U$. The first zero
of the next $V$-block supplies either the closing $U$ of the current
positive entry or the $i$ of the current zero. The final entry of
$Z$ is zero and has no following $V$-block. The B rank queues are $Y$ and
$\rev(Y)$. The initial $C_{6,0}$ reads the physical words $(4\,10,9)$
and swaps the roles. The terminal $C_{7,11}$ restores their order.
Therefore its contribution to the first physical source appends no
letter after $11$, while its contribution to the second appends $4$
after $10$. These two one-block judgements follow directly from the
rules $17455430\to692369$ and $709032\to17373884$, respectively.

To make all intervening words visible, abbreviate
\[
 a_m=9\phi_3(P_m),\qquad b_m=H_m\,11,\qquad C=4\,10,
 \qquad f_r=4G_{3A}(r),\qquad g_r=G_{4A}(r)10.
\]
For $r\ge-1$ the gap and constant-edge definitions give
\[
 f_r=C\prod_{m\in\rho_{r+2}}a_m,\qquad
 g_r=4\left(\prod_{m\in\pi_{r+2}}b_m\right)10 .
\]
In particular $a_1=9\,0\,3$ and $b_1=5\,0\,2\,11$.
The four A targets $o,U,D,i$ are respectively
$f_r,C,4b_1 10 a_1,g_r$.
After the initial physical swap the B tokens append the following words,
where $P$ and $S$ denote the first and second physical sources, and
$G_{3A}^{\circ}(r)$ denotes $G_{3A}(r)$ without its initial $10$.
\[
\begin{array}{c|cc}
 &P&S\\\hline
 o(r)&G_{3A}^{\circ}(r-1)C&\phi_3(P_{r+2})9\\
 U   &4b_1&0\,3\,C\\
 D   &10&9\\
 i(r)&b_{r+2}&G_{4A}(r-1)10
\end{array}
\]
This follows from the token zipper table by exchanging its two
columns and deleting each shared initial letter.

\paragraph{First physical source.}
The initial $C$ and the first $n-1$ full $o$-tokens use $O_{n-1}$
from the forward queue. Their words telescope to
\[
 \left(\prod_{r\in O_{n-1}} f_{r-1}\right)C.
\]
They give exactly the initial A tokens $o^{n-1}U$, since the first
$n-1$ entries of $X$ are $E_{n-1}=O_{n-1}-1$
(also for $X=A_1$). The following B $U$ now appends $4b_1$.
Correspond each positive $z$ of $Z$ with $y=z-1$ in $Y$.
Its opening B $D$ appends $10$. Each zero of $Z$ before this opening
has appended one $b_{r+2}$ from the reverse queue. Consequently the
word since the preceding B $U$, up to and including this opening $D$,
is
\[
 4b_1\left(\prod b_{r+2}\right)10.                 \tag{C.2}
\]
We determine its list of $r$'s explicitly. The initial positives of
$E_n$ are followed by its final zero. In each subsequent $U(w)$,
$\pi_w$ consists of $w$ consecutive positives, followed by one zero.
Thus the only positive $z$ preceded by zeros is the first positive
of a nonempty $\pi_w$, namely $z=1$, so $y=0$.
Write
\[
 Y=0^{d_0}w_1 0^{d_1}\cdots w_k0^{d_k},\qquad w_j>0 .
\]
There are $1+d_{j-1}$ zeros before the first positive of $\pi_{w_j}$,
and $1+d_k$ after the last positive in $Z$.
Lemma~\ref{lem:port-rank-skeleton}, applied to $(Y,X)$, gives
\[
 (d_0,\ldots,d_{k-1})=\rev(X),\qquad d_k=n-1.
 \tag{C.3}\label{eq:b-gap-reversal}
\]
The $i$-queue $\rev(Y)$ is the concatenation of the words
\[
 \rev(V(t))=\pi_t0\qquad(t\in\rev(X)),
\]
followed by $\rev(O_{n-1})$. By
\eqref{eq:b-gap-reversal}, the $j$-th such word has the required
length $1+d_{j-1}$. At a zero $y$ of $Y$, the product in (C.2)
therefore has index list $(\pi_t+2)\,2$, and $b_1$ supplies the
initial $1$. Since $\pi_{t+2}=1(\pi_t+2)2$, this is precisely
the A one-token $g_t$ assigned from $\rev(X)$. At a positive $y$
the product in (C.2) is empty.

At that positive $y$, write $x=y-1$. The following $y$ B zero-tokens
consume $V(x)=0\rho_x$ from the forward queue, since its remainder
is $V(X)$ and the positive $y$ enumerate $X+1$. Their first rank
is zero. Together with (C.2) this gives
\[
 (4b_1 10)a_1\left(\prod_{p\in\rho_x}f_{p-1}\right)C .
\]
This is the A group $D o^{y-1}U$ in \eqref{eq:port-b-a-path}.
Indeed $\rho_x-1=U(x-1)$ for $x>0$, and the list is empty for
$x=0$. Over all $x$ their concatenation is $U(F(X))$, precisely
the rest of the A forward queue after $E_{n-1}$; at $n=2$ this
rest is empty. The last B $U$ supplies $4b_1$. The final $1+d_k=n$
zeros of $Z$ give $d_k$ full $i$-tokens, since the final entry
has no following $V$-block.
The final full tokens take the remaining ranks
$\rev(O_{n-1})=(1,3,\ldots,2n-3)$. The terminal half-loop appends
no letter to this source. The resulting bridge fragment is
\[
 4b_1 b_3\cdots b_{2n-1}=\pref_{11}(\Abridge_n).
\]
Thus the first physical source is
$\AEpoch(n-1)\pref_{11}(\Abridge_n)$.

\paragraph{Second physical source.}
Set $a_m=9\phi_3(P_m)$ as above. The initial $9$, the ranks
$O_{n-1}$ on the first $n-1$ full $o$-tokens, and the following $U$
give
\[
 a_{2n-1}a_{2n-3}\cdots a_1 C = \suf_9(\Bbridge_n)C .
\]
For a positive $z=y+1$ of $Z$, the opening $D$, the $y$ zero-tokens
and their closing $U$ give
\[
 \left(\prod_{r\in V(y-1)}a_{r+2}\right)a_1 C
     \quad (y>0),\qquad a_1 C \quad(y=0).
\]
In the positive case $x=y-1$ and
$(V(x)+2)1=2(\rho_x+2)1=\rho_{x+2}$. Hence the product before the
last $C$ is exactly the part of $f_x$ following its own initial $C$.
Intervening zeros of $Z$ provide single full $g_{r-1}$ words. They
occur precisely before the cells $y=0$, by the zero-run calculation
above. Their assigned group is $\pi_t0$ for the corresponding
$t\in\rev(X)$. This gives
\[
 \left(\prod_{r\in\pi_t}g_{r-1}\right)g_{-1}a_1 C
 =
 \left(\prod_{r\in\pi_t}g_{r-1}\right)
    (4b_1 10a_1)C .                                           \tag{C.4}
\]
After the final positive of $Z$, the remaining group supplies
$g_0g_2\cdots g_{2n-4}$, followed by the final $4$.

To identify the order, keep the last $C$ pending. Initially this is
the $C$ following $\suf_9(\Bbridge_n)$. At a positive $y$, the
product of $a$'s completes the pending $C$ to $f_{y-1}$ and leaves
another $C$ pending. At a zero $y$, emit the pending $C$ as the
A direction token $U$, append the $t$ one-tokens with ranks
$\pi_t-1$, then append the A direction token $D$ from (C.4).
Again $C$ remains pending. After the last $y$, emit the pending
$C$ as $U$, the final $n-1$ one-tokens, and $4$.

Indeed the zero entries of $Y=O_{n-1}V(X)$ are in one-to-one
correspondence with $x\in X$. Its positive runs, including empty
runs, have lengths $n-1$ initially, and then $x$ immediately
after the zero for that $x$. In the A path, the initial positive
run gives $o^{n-1}U$. Its $j$-th positive $w_j$ in $Y$
gives $D o^{w_j-1}U$, preceded by $d_{j-1}$ full one-tokens
and, after the last such group, followed by $d_k$ one-tokens.
Here $w_j-1$ is the $j$-th entry of $X$ and $d_{j-1}$ the $j$-th
entry of $\rev(X)$ by \eqref{eq:b-gap-reversal}. Hence the
pending-letter procedure puts every direction token at its required
position. The forward ranks on the $o$-tokens are the positive
entries of $Y$ minus one, hence $X$. On the $i$-tokens
$\pi_t-1=0\rho_{t-1}=\rev(U(t-1))$ for $t>0$, and the list
is empty for $t=0$. Concatenating for $t\in\rev(X)$ gives
$\rev(U(F(X)))$. Appending the boundary ranks $\rev(E_{n-1})$
therefore gives the entire reverse queue $\rev(X)$, including
$n=2$ when $U(F(X))$ is empty. The complete second source is
$\suf_9(\Bbridge_n)\AEpoch(n-1)4$.

\section{Equivalence with the archived rank generator}\label{app:rank-insertion}

Here we recover the same-family insertion algorithm of the archived
\path{code/rank_words.py} from Lemma~\ref{lem:port-rank-skeleton}.
Call a letter of a new rank word high if it is at least \(2\).
Reading its high letters and subtracting two gives the old word
by \(F^2\). We identify the intervening words of zeros and ones.
In this appendix \(0^j\) denotes a string of rank zeros.

\begin{proposition}[A insertion]
For \(n\ge1\), traverse \(A_n\) and output \(x+2\) at every letter
\(x\). After each old zero except the last, append \(P_r\), with the
parameters \(r\) given by the positive entries of \(A_n\) in reverse
order. After the last zero append
\[
 \left(\prod_{j=1}^{n}0^{2j-1}1\right)0^{n+1}.
\]
The resulting word is \(A_{n+1}\).
\end{proposition}
\begin{proof}
Put \(X=A_n,\ P=B_n,\ Y=B_{n+1}=\mathcal I_n(X)\). The target
is \(E_{n+1}U(Y)\). Its initial run of high letters and every high run
in \(U(y)\), \(y\ge2\), end in their unique \(2\). Hence only an
old zero can be followed by an insertion. The intervening low
letters in \(Y\) are described by reading a run of \(d\) zeros and
the next positive letter in \(X\). The leading zero counts in
\(\mathcal I_n(X)\) for these \(d+1\) entries are a consecutive
part \((z_1,\ldots,z_{d+1})\) of the reverse queue \(\rev X\).
Thus, before the next high \(y\), the low part of \(Y\) is
\[
 0^{z_1}1\,0^{z_2}1\cdots0^{z_d}1\,0^{z_{d+1}} .
\]
After the preceding high run in the target there is an outstanding
zero (the end of \(E_{n+1}\) or of \(U(y)\)). The next high group
in \(U(y')\) is preceded by its leading \(1\). Since \(U(0)=0\)
and \(U(1)=10\), the intervening target word is
\(\prod_{j=1}^{d+1}0^{z_j+1}1\).
The successive \(d\)'s before positives of \(X\) are \(\rev P\).
The respective parts of \(\rev X=\rev(U(P))\rev(E_n)\) are
\(0\rho_d\). Consequently \(z+1=1(\rho_d+1)=\pi_{d+1}\)
in componentwise notation, and the insertion is \(P_{d+1}\).
The parameters \(d+1\) are the reversed positive entries of \(X\)
because \(F(X)=P\).

The terminal \(n\) zeros of \(X\) consume the remaining queue
\(\rev E_n=(0,2,\ldots,2n-2)\). There is no next high group,
and \(Y\) has \(n\) trailing zeros. Applying \(U\) and including
the outstanding zero therefore gives
\(\bigl(\prod_{j=1}^n0^{2j-1}1\bigr)0^{n+1}\). For \(n=1\)
there are no interior insertions and the same tail calculation applies.
\end{proof}

\begin{proposition}[B insertion]
For \(n\ge2\), traverse \(B_n\) and output \(x+2\) at every letter.
After each old \(1\) append \(Q_1=10\). Form the list of positive
entries of \(B_n\) in reverse order, replace each \(1\) in that
list by \(0\), and discard its first entry. After each old zero
except the last append \(Q_r\) for the next parameter \(r\) on
this list, with \(Q_0\) empty. After the last zero append
\[
 1\left(\prod_{j=1}^{n-1}0^{2j}1\right)0^n .
\]
The resulting word is \(B_{n+1}\).
\end{proposition}
\begin{proof}
Put \(P=A_{n-1},\ X=B_n=O_{n-1}V(P),\ Y=A_n=\mathcal I_n(X)\).
The target is \(O_nV(Y)\). Write \(p_1,\ldots,p_k\) for the
successive entries of \(P\). The ordered positive entries of \(X\)
are \(x_i=p_i+1\), and the numbers of zeros immediately preceding
them are \(d_{i-1}=p_{k+1-i}\). The terminal run has length \(d_k=n-1\).
Both \(d_0=0\) and \(x_k=1\), since \(P\) ends in zero.

The high letters of \(O_n\) end in \(3\); its final \(1\) and the
initial \(0\) of the first \(V(y)\) together give \(Q_1\) after
this old \(1\). Each high \(y=x_i+1\) in \(Y\) gives one further
high group in \(V(y)=0\rho_y\). This group starts with \(2\);
when \(x_i>1\) it has at least one further letter and ends with
\(3\), the only \(3\) in the group. When \(x_i=1\) it consists
of the single \(2\). These groups account in order for all \(k\)
old zeros, one at each first \(2\).

Between the high groups for \(x_i\) and \(x_{i+1}\), put \(d=d_i\).
The \(d\) intervening zero letters of \(X\) and the next positive
consume \(d+1\) entries \(z\) of \(\rev X\) in \(\mathcal I_n(X)\).
They are exactly \(\pi_d0\): the groups of
\(\rev X=\rev(V(P))\rev(O_{n-1})\) have this form for
\(d_0,d_1,\ldots,d_{k-1}\), and the first group, for \(d_0=0\),
was used at the first high letter of \(Y\).
Since \(V(0)=0,\ V(1)=01\), the terminal \(1\) of the preceding
\(\rho_y\), the intervening low letters, and the leading \(0\)
of the next \(V(y')\) yield the packet
\[
 \prod_{j=1}^{d+1}1\,0^{z_j+1}=Q_{d+1},
 \qquad z+1=(\pi_d+1)1=\rho_{d+1}.
\]
In \(X=O_{n-1}V(P)\) each positive run ends with its unique \(1\).
Thus for \(i<k\), \(x_i>1\) gives \(d_i=0\). In this case the
packet is \(Q_1\) after the old \(1\) ending the group and the
insertion at its initial old zero is empty. If \(x_i=1\), then
\(d_i\ge1\); this time \(Q_{d_i+1}\) follows the old zero itself.
The parameters before replacement are \(d_i+1=p_{k-i}+1\)
for \(1\le i<k\). They are the reversed positive list after its
first entry \(d_0+1=1\) is discarded, which proves exactly the
stated queue rule.

Finally \(x_k=1\) supplies the last old zero. The remaining \(n-1\)
zero letters of \(X\) consume \(\rev O_{n-1}=(1,3,\ldots,2n-3)\);
\(Y\) has \(n\) trailing zeros. The final \(1\) from the high
group and those low letters therefore give
\(1(\prod_{j=1}^{n-1}0^{2j}1)0^n\).
\end{proof}

\section{Certificate inventory}

The v4 verification archive \cite{MantovanelliReturnWordSupplementV4}
organizes the proof data and programs as follows.
\begin{itemize}
\item \path{kernel_data/} contains the 13 return words, all 92 synchronized
states, the 122 transition rules, and the cursor and role tables.
\item \path{code/} contains the word generators and the kernel checkers.
\item \path{data/} contains the reference audit tables and summaries.
\item \path{anc/check_v4.py} checks the additional finite claims of this revision.
\item \path{manuscript/} contains the paper and its source files.
\end{itemize}
The entry point \path{reproduce.py} verifies the checksums in
\path{SHA256SUMS.txt}, invokes \path{run_all_checks.py} for the ten
kernel checks, and then runs \path{anc/check_v4.py} against the same tables.
The kernel report and the v4 check log are written to the result directory
identified at the start of the run.

\begingroup
\raggedright
\small
\bibliographystyle{unsrt}
\bibliography{references}

\begin{thebibliography}{10}

\bibitem{Hofstadter1979}
Douglas~R. Hofstadter.
\newblock {\em Gödel, Escher, Bach: An Eternal Golden Braid}.
\newblock Basic Books, 1979.

\bibitem{Pinn1999}
Klaus Pinn.
\newblock Order and chaos in {Hofstadter}'s {Q}-sequence.
\newblock {\em Complexity}, 4(3):41--46, 1999.

\bibitem{Tanny1992}
Stephen~M. Tanny.
\newblock A well-behaved cousin of the {Hofstadter} sequence.
\newblock {\em Discrete Mathematics}, 105:227--239, 1992.

\bibitem{Alkan2017}
A.~Alkan and N.~Fox.
\newblock On {Hofstadter}-type sequences.
\newblock {\em Complexity}, 2017:1--10, 2017.

\bibitem{Alkan2018}
A.~Alkan.
\newblock On a generalization of {Hofstadter}'s {Q}-sequence: A family of
  chaotic generational structures.
\newblock {\em Complexity}, page Article 8517125, 2018.

\bibitem{Fox2016LinearRecurrent}
Nathan Fox.
\newblock Finding linear-recurrent solutions to {Hofstadter}-like recurrences
  using symbolic computation.
\newblock {\em arXiv preprint}, 2016.
\newblock arXiv:1609.06342.

\bibitem{Fox2018NewApproachQ}
Nathan Fox.
\newblock A new approach to the {Hofstadter} {Q}-recurrence.
\newblock {\em arXiv preprint}, 2018.
\newblock arXiv:1807.01365.

\bibitem{CelayaRuskey2012Undecidable}
M.~Celaya and Frank Ruskey.
\newblock An undecidable nested recurrence relation.
\newblock In {\em How the World Computes}, volume 7318 of {\em Lecture Notes in
  Computer Science}, pages 138--149. Springer, 2012.

\bibitem{OEISA394051}
OEIS~Foundation Inc.
\newblock Entry a394051 in the on-line encyclopedia of integer sequences.
\newblock \url{https://oeis.org/A394051}, 2026.
\newblock Accessed: 2026-03-31.

\bibitem{Mantovanelli2026}
Marco Mantovanelli.
\newblock A dyadic frequency law for a perturbed {Hofstadter} {Q}-recursion.
\newblock {\em arXiv preprint}, 2026.
\newblock arXiv:2603.16111v3, 18 July 2026.

\bibitem{Cloitre2026PerturbedQ}
Beno{\^i}t Cloitre.
\newblock The {Mantovanelli--Hofstadter} sequence.
\newblock {\em arXiv preprint}, 2026.
\newblock arXiv:2604.06237v3.

\bibitem{MantovanelliReturnWordSupplementV4}
Marco Mantovanelli.
\newblock Verification supplement for {Certified Return-Word Induction for a
  Perturbed Hofstadter Recursion}.
\newblock Zenodo, \url{https://doi.org/10.5281/zenodo.24441831}, 2026.
\newblock Version 4.0.0.

\bibitem{Cobham1972}
Alan Cobham.
\newblock Uniform tag sequences.
\newblock {\em Mathematical Systems Theory}, 6:164--192, 1972.

\end{thebibliography}
\endgroup

\end{document}